\documentclass[11pt]{article}

\usepackage{style}

\title{Extending the Wasserstein metric to positive measures}
\author{H. Leblanc\thanks{\url{hugo.leblanc@univ-amu.fr}}, T. Le Gouic\thanks{\url{thibaut.le_gouic@math.cnrs.fr}}, J. Liandrat\thanks{\url{jacques.liandrat@centrale-marseille.fr}} and M. Tournus\thanks{\url{magali.tournus@centrale-marseille.fr}}
\thanks{Centrale Marseille, I2M, UMR 7373, CNRS, Aix-Marseille univ., Marseille, 13453, France} }
\date{\today}
\begin{document}

\maketitle

\begin{abstract}
We define a metric in the space of positive finite positive measures that extends the 2-Wasserstein metric, i.e. its restriction to the set of probability measures is the 2-Wasserstein metric.
We prove a dual and a dynamic formulation and  extend the gradient flow machinery of the Wasserstein space.
In addition, we relate the barycenter in this space to the barycenter in the Wasserstein space of the normalized measures.
\end{abstract}

\section{Introduction}

The 2-Wasserstein metric $W_2$ is a metric based on optimal transport on the space of probability measures on $\R^d$.
The resulting metric space is called the Wasserstein space.
It is defined by
\begin{equation}\label{eq:W2}
W_2^2(\mu,\nu) = \min_{\pi \in \Pi(\mu,\nu)} \int \vert x-y \vert^2 \ud\pi(x,y),
\end{equation}
where $\Pi(\mu,\nu)$ denotes the set of probability measures on $\R^d \times \R^d$ with marginals $\mu$ and $\nu$.
This formulation is due to Kantorovitch \cite{kantorovich1942translocation,kantorovich1948problem} and based on the original work of Monge \cite{monge1781memoire}.
Kantorovitch also famously introduced the dual formulation
\begin{equation}\label{eq:W2_dual}
W^2_2(\mu,\nu) = \sup \int \phi \ud\mu + \int \psi \ud\nu,
\end{equation}
where the supremum is taken among the pairs of functions $(\phi,\psi) \in \mc{C}_b(\R^d)^2$ satisfying $\phi(x)+\psi(y)\le \vert x-y\vert^2$. This formulation has been central to many other results, among which the Brenier theorem stating that in regular cases, the minimum in the primal problem is attained for a coupling supported on the graph of a function: the Brenier or Monge map \cite{brenier1991polar}.

A dynamic formulation due to Benamou and Brenier \cite{benamou2000computational}, introduced for numerical purposes, reinterprets the Wasserstein metric as a minimization of a kinetic energy functional
\begin{equation}\label{eq:W2_dyn} 
W_2^2(\mu,\nu) = \inf \int_0^1\int \vert v_t \vert^2 \ud\mu_t \ud t.
\end{equation}
Here, the infimum is taken among the probability measures valued maps $\mu_t:[0,1] \to \mc{P}_2(\R^d)$ and vector field valued maps $v_t:[0,1] \to L^2(\mu_t)$ satisfying the weak transport partial differential equation
\[ 
\partial_t \mu_t + \dive (v_t \mu_t) = 0, \quad \mu_0=\mu, \; \mu_1=\nu. 
\]
This formulation formally endows the Wasserstein space with a Riemannian-like structure.
It thus benefits from a pseudo-tangent structure, that gives rise to the reinterpretation of many PDEs as gradient flows on the Wasserstein space, a trend that initiated with the seminal work of Jordan, Kinderlehrer and Otto \cite{jordan1998variational}.
The resulting geodesic structure on the Wasserstein space makes it a natural object to study the Fréchet mean, a.k.a. barycenters on those space.
The Wasserstein barycenters were introduced in \cite{agueh2011barycenters} and have since been applied, for instance, in image processing \cite{rabin2012wasserstein,bonneel2015sliced,simon2020barycenters} and Bayesian inference \cite{backhoff2022bayesian}, and have attracted theoretical interest as well.

We refer to the monographs \cite{Villani2003,Ambrosio2005,santambrogio2015optimal} for further background on the topic.

The numerical complexity of the computation the $W_2$ metric is still a very active research topic, with frequently new algorithms proposed.
Among the many algorithms proposed, we mention the simplex method which arises from the field of linear programming, the Sinkhorn algorithm \cite{cuturi2013sinkhorn,thibault2021overrelaxed,altschuler2017near}, and more recently the Back-and-Forth method \cite{jacobs2020fast} that shows striking visual results.
We refer the reader to \cite{COTFNT} for further details.

The properties of $W_2$ and the advances in numerical computation have made $W_2$ a very efficient multidisciplinary tool to compare and interpolate data. 
Many of the uses in the field of computer science, for image processing, color transfer \cite{rabin2014adaptive} rely on the existence of the Monge map; style transfer \cite{mroueh2019wasserstein} relies on the existence of geodesic to interpolate artistic styles; in machine learning the dual formulation of $W_2$ metric gives a more stable method to train a GAN \cite{liu2019wasserstein}; in natural language processing, the Wasserstein metric can be used to learn a bag-of-words representation ; in biology, $W_2$ is used to study the trajectory of the differentiation of a cell \cite{schiebinger2019optimal} using geodesic interpolation; in economics \cite{ponti2022unreasonable}, the $W_2$ metric allows for a deeper analysis of the customer experience for different stores by comparing the key performance indicator distributions.

While being central to the definition of the Wasserstein metric, the requirement of two measures to be of same mass is often a strong limitation.
For instance, in point and shape matching between point cloud \cite{shen2021accurate,bonneel2019spot}, or in \cite{schiebinger2019optimal} where cell counts are not constant and need to be matched over time.
This is often dealt with in practice by normalizing the measures, to the price of loosing the information of the total mass and the departing from the theoretical framework of the Wasserstein metric.
This limitation is the starting point of goal of this paper: the search for a meaningful metric that extends $W_2$ to the broader set of positive finite measures. 
This leads also to the simple question : is there a non trivial metric on finite positive measures whose restriction on probability measures is the Wasserstein metric?

\paragraph{Related work}
So far, the problem has been dealt with differently and other metrics based on optimal transport have been introduced.
The theory of Unbalanced (or partial) Optimal Transport (UOT) has been developed to that aim, and is based on classical optimal transport.
Different approaches based on the multiples formulation of the classic OT are used to define meaningful substitutes of $W_2$ for the broader space of positive finite measures.
Here we describe two of them, the Entropy-Transport and the dynamic formulations.

To define the Entropy-Transport ($\ET$) formulation \cite{piccoli2014generalized, gallouet2021regularity, liero2018optimal, chizat2018unbalanced} in a broad setting, we first recall the Csiszar $f$-divergences. 
Let $\mu,\nu$ be positives measures on $\R^d$, $f:\R_+ \to [0,+\infty]$ be a convex, lower semi-continuous function such that $f(1)=0$ and set $f_{\infty}'(1) = \lim\limits_{t \to +\infty} \frac{f(t)}{t}$, then the $f$-divergence is
\begin{equation*}\label{eq:fDiv}
D_f(\mu \vert \nu) = \int f\left(\frac{\ud\mu}{\ud\nu}\right) \ud \nu + f_{\infty}'(1)\nu^{\perp}(\R^d), \text{with the Lebesgue decomposition }  \mu = \frac{\ud\mu}{\ud\nu}\nu + \nu^\perp.
\end{equation*}
For $c(x,y): \R^d \times \R^d \to [0,+\infty]$ a lower semi-continuous cost function such that $c(x,x)=0$ for all $x\in\R^d$,  the Entropy-Transport formulation is a relaxation of the primal formulation of $W_2$ (equation~\eqref{eq:W2}): a mass transportation using the classical OT while relaxing the constraint on the marginals of the transport measure by penalizing the creation and destruction of mass with Csiszar f-divergences
\begin{equation}\label{eq:UOT_ET}
\ET(\mu,\nu) = \inf\limits_{\pi} \bigg\{ D_{f}(P_0\#\pi \vert \mu) + D_{f}(P_1\#\pi \vert \nu) + \int c(x,y) \ud\pi(x,y) \bigg\},
\end{equation}
where the infimum is taken among positive measures and  $P_0\#\pi$ and $P_1\#\pi$ are the first and second marginals of $\pi$.
Choosing wisely $f$ and costs $c$ give rise to different metrics in the set of positive measures.

The dynamic approach ($\Dyn$) of UOT is based on the dynamic formulation of $W_2$ \eqref{eq:W2_dyn} and goes as follows.
Given a convex, continuous function $L(v,g):\R^d \times \R \to \R$ which is minimal at $L(0,0)=0$ --- called the Lagrangian function, the dynamic UOT formation is given by minimizing the action functional
\begin{equation}\label{eq:UOT-BB}
\Dyn(\mu,\nu)=\inf \int_0^1 \int L(v_t,g_t)\ud\mu_t \ud t,
\end{equation}
where $g_t$ is a source term accounting for the creation and destruction of mass, the infimum is taken among the triplets $(v_t,g_t,\mu_t)$ satisfying the transport equation with source
\[
\partial_t\mu_t + \dive (v_t \mu_t) = g_t \mu_t, \quad \mu_0=\mu, \; \mu_1 = \nu.  
\] 
The equivalence under some assumptions between these two formulations of UOT has been proven in \cite{liero2018optimal,chizat2018unbalanced}.
See for instance \cite{chizat2018interpolating, lombardi2015eulerian, maas2015generalized} for further details on UOT.

Among the UOT based metrics, the one that has attracted the most interest if the  Hellinger-Kantorovich metric $\HK$ (also called the Wasserstein-Fisher-Rao or Hellinger-Fisher-Rao), introduced simultaneously in \cite{chizat2018interpolating,kondratyev2016new,liero2018optimal}.
This metric can be defined using both UOT formulations described above 
\[
\HK^2(\mu,\nu) = \ET(\mu,\nu) = \Dyn(\mu,\nu),
\]
where the Csiszar $f$-divergences in the Entropy-Transport formulation \eqref{eq:UOT_ET} is defined using $f(t)=t\log(t)+1-t$ and  the transport cost $c(x,y) = -2\log \big( \cos(\vert x-y\vert \wedge \frac{\pi}{2}) \big)$.
The Lagrangian function $L(v,g)$ in the dynamic formulation \eqref{eq:UOT-BB} of $\HK$ is $L(v,g)= \vert v \vert^2 + \frac{1}{4}\vert g \vert^2$.

The $\HK$ metric is frequently used as a $W_2$-like metric to compare positive measures in the sense that $\HK$ inherits from the multiple formulations that $W_2$ enjoys.
However,   $\HK$ does not extend $W_2$ to positive measures, since for most probability measures $\mu,\nu \in \mc{P}_2(\R^d)$ we have $\HK(\mu,\nu) < W_2(\mu,\nu)$; see Figure~\ref{fig:WOPGeod2}.

\begin{figure}[h]
    \centering
    \begin{subfigure}[b]{0.4\linewidth}
    \includegraphics[width=\linewidth]{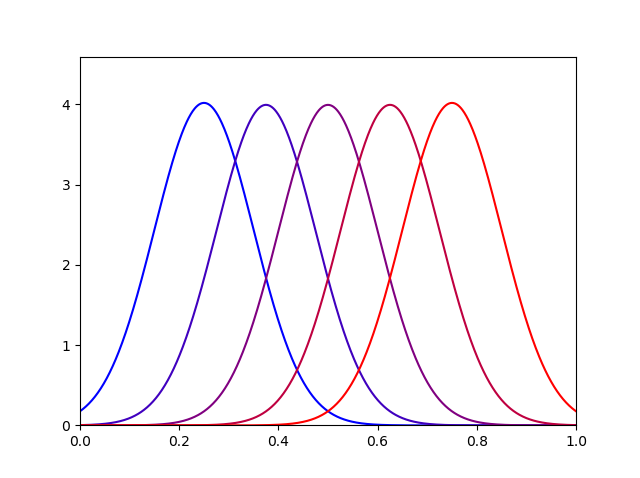}
    \caption*{$\WOP$ geodesic}
    \end{subfigure}
    \begin{subfigure}[b]{0.4\linewidth}
    \includegraphics[width=\linewidth]{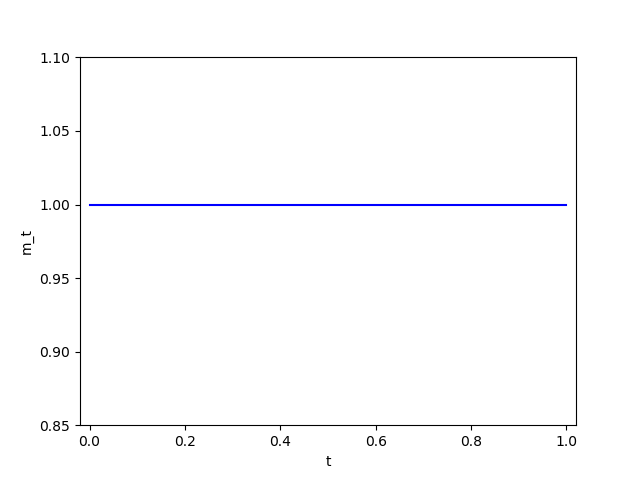}
    \caption*{ mass $m_t$ of the $\WOP$ geodesic}
    \end{subfigure}

    \begin{subfigure}[b]{0.4\linewidth}
    \includegraphics[width=\linewidth]{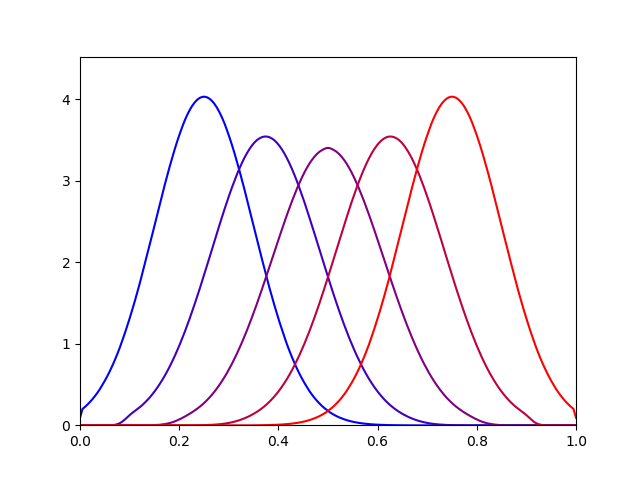}
    \caption*{$\HK$ geodesic}
    \end{subfigure}
    \begin{subfigure}[b]{0.4\linewidth}
    \includegraphics[width=\linewidth]{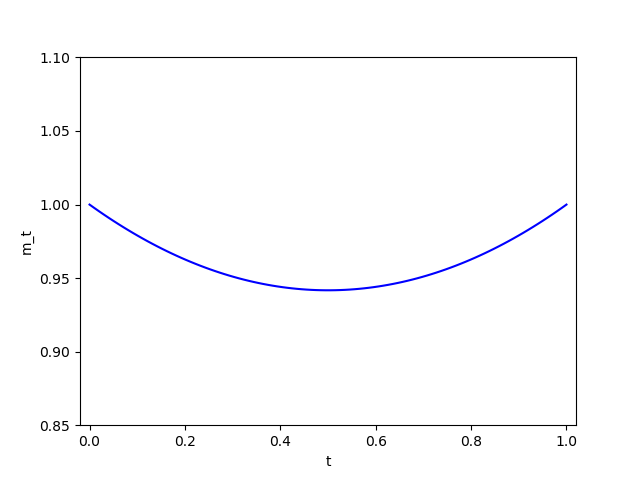}
    \caption*{ mass $m_t$ of the $\HK$ geodesic}
    \end{subfigure}
    \caption{Comparison of geodesics between Gaussian measures for the $\WOP$ and $\HK$ metric}
    \label{fig:WOPGeod2}
\end{figure}

A central observation at the origin of present work, is that in all these metrics, the space of probability measures is not geodesically convex, i.e. the geodesic between any two distinct probability measures does not lie in the space of probability measures.
In fact, we prove in Section~\ref{sec:further} that it is impossible to extend $W_2$ on $\R^d$ with the UOT framework.
This leads to the following question : is there a non trivial metric defined on all positive measures whose restriction on probability measures is the Wasserstein distance?

We answer with the affirmative and propose a metric that preserves the multiple formulations, whose geodesics are 'meaningful' in a way we develop, and that inherits the algorithmic computations of the $W_2$ metric, interpolations and barycenters.

\paragraph{Outline}
The outline of the paper is as follows.
In Section~\ref{sec:WOP}, we define a metric $\WOP$ on positive measure that $W_2$.
We prove a dual formulation and study its topological and geometric properties.
Section~\ref{sec:Dynamic} is devoted to a dynamic formulation of $\WOP$ and the subsequent gradient flows.
In Section~\ref{sec:barycenter}, we describe and analyze the barycenter associated to $\WOP$, and Section~\ref{sec:further} discusses further results and developments. 

\paragraph{Notation}
\begin{itemize}
    \item $\mc{P}_2(\R^d)$ is the set of positive finite measures with second order moment,
    \item $\mc{M}(\R^d)$ is the set of positive finite measures with second order moment,
    \item $\mc{M}_K(\R^d)$ is the set of positive finite measures with control on the second order moment (see Definition~\eqref{def:Mk}),
    \item $m_\mu$ is the total mass of $\mu$,
    \item $\bar\mu$ for $\mu/m_\mu$ if $m_\mu>0$ and $\delta_{x_0}$ otherwise,
    \item $M_{x_0}(\mu)=\int d^2(x,x_0)\ud\mu(x)$,
    \item $T_a(x) = a(x-x_0) + x_0$,
    \item $T\#\mu$ is the pushforward of $\mu$ by $T$,
    \item $\Pi(\mu,\nu)$ is the set of transportation plans between $\mu$ and $\nu$,
    \item $\Pi_o(\mu,\nu)$ is the set of optimal transportation plan between $\mu$ and $\nu$  for the cost function $\vert x-y\vert^2$,
    \item $C_c^\infty(\Omega)$ is the set of infinitely differentiable functions with compact support in $\Omega$.
\end{itemize}

\section{The $\WOP$ metric and its basic properties}\label{sec:WOP}
As mentioned in the introduction, we aim at defining a metric on the space of finite positive measures $\mc{M}(\R^d)$ that coincides with the 2-Wasserstein metric on the set of probability measures $\mc{P}_2(\R^d)$.
Due to the natural conic structure of the space of finite positive measures, our metric is inspired by the cone on $\R^d$ (see Figure \ref{fig:my_label}).

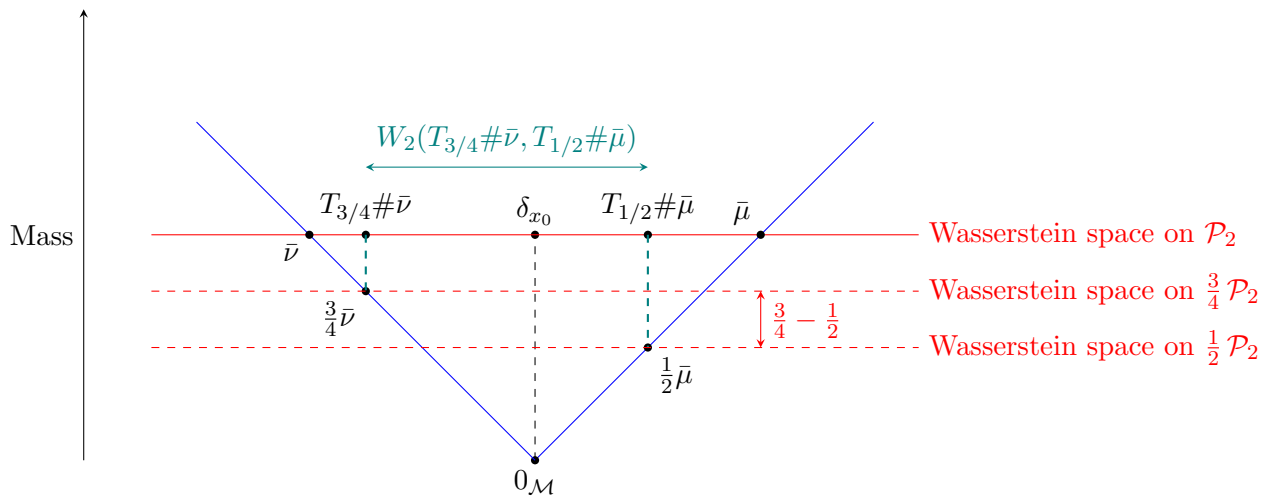
\begin{figure}[h]
    \centering
\begin{tikzpicture}[scale=3]
\draw node[circle,fill,scale=0.3] (a) at (0,1) {};
\draw node[anchor=south] (b) at (0,1) {$\delta_{x_0}$};
\draw[dashed] (0,0) -- (0,1);

\draw[blue] (1.5,1.5) -- (0,0);
\draw node[circle,fill,scale=0.3] () at (0,0) {};
\draw node[anchor=north] () at (0,0) {$0_{\mc M}$};
\draw[red] (-1.7,1) -- (1.7,1) node[anchor=west]{Wasserstein space on $\mc P_2$};

\draw node[circle,fill,scale=0.3] (a) at (1,1) {};
\draw node[anchor=south east] (b) at (1,1) {$\bar\mu$};

\draw node[circle,fill,scale=0.3] (c) at (0.5,0.5) {};
\draw node[anchor=north west] (d) at (0.5,0.5) {$\frac{1}{2}\bar\mu$};

\draw node[circle,fill,scale=0.3] (c) at (0.5,1) {};
\draw node[anchor=south] (d) at (0.5,1) {$T_{1/2}\#\bar\mu$};

\draw[thick,teal,dashed] (0.5,1) -- (0.5,0.5);

\draw[blue] (-1.5,1.5) -- (0,0);
\draw node[circle,fill,scale=0.3] (a) at (-1,1) {};
\draw node[anchor=north east] (b) at (-1,1) {$\bar\nu$};

\draw node[circle,fill,scale=0.3] (c) at (-0.75,0.75) {};
\draw node[anchor=north east] (d) at (-0.75,0.75) {$\frac{3}{4}\bar\nu$};

\draw node[circle,fill,scale=0.3] (c) at (-0.75,1) {};
\draw node[anchor=south] (d) at (-0.75,1) {$T_{3/4}\#\bar\nu$};

\draw[thick,teal,dashed] (-0.75,1) -- (-0.75,0.75);

\draw[red,dashed] (-1.7,0.75) -- (1.7,0.75) node[anchor=west]{Wasserstein space on $\frac{3}{4}\,\mc P_2$};
\draw[red,dashed] (-1.7,0.5) -- (1.7,0.5) node[anchor=west]{Wasserstein space on $\frac{1}{2}\,\mc P_2$};

\draw[-stealth] (-2,0) -- (-2,2);
\draw node[anchor=east] () at (-2,1) {Mass};

\draw[teal, stealth-stealth] (-0.75,1.3) -- (0.5,1.3);
\draw node[teal,anchor=south] () at (-0.125,1.3) {$W_2(T_{3/4}\#\bar\nu,T_{1/2}\#\bar\mu)$};

\draw[red, stealth-stealth] (1,0.5) -- (1,0.75);
\draw node[red,anchor=west] () at (1,0.62) {$\frac{3}{4}-\frac{1}{2}$};

\end{tikzpicture}
    \caption{By analogy with $\R^d$, the squared distance between $\frac{3}{4}\bar\nu$ and $\frac{1}{2}\bar\mu$ is decomposed as the sum of squared distance between $T_{3/4}\#\bar\nu$ and $T_{1/2}\#\bar\mu$ (that are on the geodesics between $\delta_{0}$ and respectively $\bar\nu$ and $\bar\mu$) and the squared distance between the space of measures with total mass $1/2$ and total mass $3/4$ - set to $(1/2-3/4)^2$.
    This gives the metric $\WOP$ of definition~\ref{def:Wass}}
    \label{fig:my_label}
\end{figure}

\begin{definition}[Wasserstein meetric On Positive measures]\label{def:Wass}
Given $x_0 \in \R^d$, let $(\mu,\nu) \in \mc{M}(\R^d)^2$.
The Wasserstein On Positive measures $(\WOP)$ metric is defined by
\[
\WOP^2(\mu,\nu) = (m_\mu- m_\nu )^2 + W_2^2( T_{m_\mu}\#\bar{\mu} ,T_{m_\nu}\#\bar{\nu} )
\]
where $T_a(x) = a(x-x_0)+x_0$, $x_0$ is called the reference point.
When the value of $x_0$ is relevant we will use the notation $\WOP_{x_0}$ for $\WOP$.
\end{definition}

On can check that $\WOP$ is a metric.
\begin{theorem}[$\WOP$ is a metric]\label{thm:Wmetric}
    $\WOP$  is a metric on the space $\mc{M}(\R^d)$.
\end{theorem}
\begin{proof}
 Positivity and symmetry are immediate. For positive-definiteness :
\[\WOP(\mu,\nu) = 0 \iff \left\{ \begin{array}{ll} 
 m_\mu = m_\nu\\
 T_{m_\mu}\#\bar{\mu} = T_{m_\nu}\#\bar{\nu}\\
 \end{array} \right. \iff \mu = \nu
 \]
 For the triangle inequality, using that $\vert . \vert$ and $ W_2$ are  metrics respectively on $\R$ and $\mc{P}_2(\R^d)$, and the Minkowski inequality, we show the following.
 \begin{align*}
    \WOP(\mu,\nu) &= \sqrt{ \vert m_\mu - m_\nu \vert^2 + W_2^2 ( T_{m_\mu}\#\bar{\mu}, T_{m_\nu}\#\bar{\nu} ) }\\
    &\le \sqrt{  \big[ \vert m_\mu - m_\eta \vert + \vert m_\eta - m_\nu \vert \big]^2 + \big[ W_2( T_{m_\mu}\#\bar{\mu}, T_{m_\eta}\#\bar{\eta} ) + W_2( T_{m_\eta}\#\bar{\eta},T_{m_\nu}\#\bar{\nu}) \big]^2 }\\
    &\le  \sqrt{ \vert m_\mu - m_\eta \vert^2 + W_2^2 ( T_{m_\mu}\#\bar{\mu}, T_{m_\eta}\#\bar{\eta} ) } + \sqrt{ \vert m_\eta - m_\nu \vert^2 + W_2^2 ( T_{m_\eta}\#\bar{\eta}, T_{m_\nu}\#\bar{\nu}) }\\
    &= \WOP(\mu,\eta) + \WOP(\eta,\nu).
\end{align*}
\end{proof}
By using a fundamental property of the Euclidean distance on $\R^d$, we can rewrite the $\WOP$ metric in a purely metric way.
On $\R^d$, for all $x,y\in \R^d$ and $a,b\in\R_+$,
    \begin{equation}\label{eq:flateq}
     |a(x-x_0) - b(y-x_0)|^2=a(a-b)|x-x_0|^2 + b(b-a)|y-x_0|^2+ab|x-y|^2, 
    \end{equation}
   and thus, the $\WOP$ metric satisfies
    \[
\WOP^2(\mu,\nu) :=  (m_\mu- m_\nu )^2 + \inf_{\pi\in\Pi\left(\bar\mu,\bar\nu\right)}\int c_{m_\mu,m_\nu}^{x_0}(x,y)\ud\pi(x,y)
    \]
    where
    \[
    c_{a,b}^{x_0}(x,y)=a(a-b)\|(x-x_0\|^2 + b(b-a)\|y-x_0\|^2+ab \|x-y\|^2.
    \]
    Note that the first two terms of $c_{a,b}^{x_0}$ do not depend on the coupling.
    In particular, this shows that the optimal coupling is the same as the one between the normalized measures $\bar{\mu}$ and $\bar{\nu}$ and so the geodesics are also the same (although at a different speed and with a different total mass) as the ones between the normalized measures.
    Now, setting
    \[
    M_{x_0}(\mu):=\int d^2(x,x_0)\ud\mu(x),
    \]
     we can rewrite $\WOP$ as
    \begin{equation}\label{eq:defbis}
    \WOP^2(\mu,\nu) =  (m_\mu- m_\nu )^2 + (m_\mu-m_\nu)(M_{x_0}(\mu)-M_{x_0}(\nu)) + m_\mu m_\nu W_2^2(\bar\mu,\bar\nu).
    \end{equation}
    \begin{remark}
    One can quickly check that the dependence of $\WOP$ on the reference point $x_0$ is only through the moments centered in $x_0$.
    In fact, the following equation holds for any two reference points $x_0$ and $y_0$
    \[
    \WOP_{x_0}^2(\mu,\nu)=\WOP_{y_0}^2(\mu,\nu)+(m_\mu-m_\nu)[M_{x_0}(\mu)-M_{y_0}(\mu)+M_{y_0}(\nu)-M_{x_0}(\nu)].
    \]
    \end{remark}
    
A significant difference between the $\WOP$ metric and the $\HK$ metric is their scaling with the total mass of the measures.
While $\HK$ scales proportionally to the square-root of the mass, $\WOP$ scales proportionally to mass.
\begin{property}[$1$-homogeneity]
The $\WOP$ metric satisfies the following 1-homogeneity property, i.e.,
for all $a \in \R_+$,
\[
 \WOP(a\mu,a\nu) = a\WOP(\mu,\nu).
\]  
\end{property}

\begin{proof}
The proof is immediate using the reformulation of the metric \eqref{eq:defbis}, using that $m_{a\mu}=a m_{\mu}$ and $M_{x_0}(a\mu)=aM_{x_0}(\mu)$.
\end{proof}

In particular, when $\mu$ is the null measure $0_\mc{M}$, this property ensures that the geodesic between a measure and the null measure lies on the rescaling of the original measure.

We now highlight the main property of this metric: it extends the Wasserstein metric to all finite positive finite measures.
This result is formulated in the following theorem.
\begin{theorem}[$\WOP$ extends the Wasserstein distance]
    The $\WOP$ metric satisfies that
    \begin{itemize}
        \item for all $\mu,\nu\in\mc P_2(\R^d)$, and $a\in\R_+$,
        \[
        \WOP(a\mu,a\nu)=a W_2(\mu,\nu),
        \]
        \item and for all $\mu\in\mc P_2(\R^d)$, and $a_0,a_1\in\R_+$,
        \[
        \WOP^2(a_0\mu,a_1\mu)=(a_0-a_1)^2(1+M_{x_0}(\bar\mu)).
        \]
    \end{itemize}
\end{theorem}
In particular, the set of probability measures is geodesically convex for the $\WOP$ metric.

\subsection{Topology of $\WOP$}

Before describing the topology induced by the $\WOP$ metric, we start the section with an observation.
Given a metric $D$ on $\mc{M}(\R^d)$ that extends the Wasserstein metric, one can expect that for any sequence $a_n \in \R_+$ converging to $0$ and any sequence $\mu_n$ of probability measures, $a_n\mu_n \to  0_{\mc{M}}$ for the topology of $D$.
However, if $D$ satisfies some homogeneity property (i.e $D(a\mu, a\nu)\approx a^p D(\mu,\nu)$), then taking $a_n \to 0$ and two sequences $\mu_n$ and $\nu_n$ of probability measure,
\[  
D(a_n \mu_n, a_n \nu_n) \approx a_n^p W_2(\mu_n,\nu_n).
\]
Thus, if $W_2(\mu_n,\nu_n)$ tends to $+\infty$ faster than $a_n^{p}$ converges to $0$, the measures $a_n \mu_n$ and $a_n \nu_n$ cannot both converge to $0_\mc{M}$. To avoid this issue, we introduce a subset of the positive measures so that $W_2(\mu_n,\nu_n)$ remains bounded.
\begin{definition}[Positive measure with control on the second-order moment]
For $K>0$ and a reference point $x_0$ for the metric $\WOP_{x_0}$, we define $\mc{M}_K(\R^d)$ as the set of positive measures such that
\[
M_{x_0}(\mu) \le Km_\mu.
\]
\label{def:Mk}
\end{definition}

On that space, the observation above cannot happen and $\WOP$ is a complete metric.
\begin{property}[Completeness]
The space $(\mc{M}_K(\R^d),\WOP)$ is complete
\end{property}

\begin{proof}
Let $\mu_n$ be a Cauchy sequence in $\mc{M}_K(\R^d)$. By completeness of the metrics $\vert\, . \,\vert$ and $W_2$, we have that $m_{\mu_n}$ converges to $m \in \R_+$ and $T_{m_{\mu_n}}\#\bar{\mu}_n$ converges to $\lambda \in \mc{P}_2(\R^d)$. From this we can show that $\mu_n$ converges.
\begin{enumerate}
    \item First case : suppose that $m > 0$, then $T_m(x)$ is bijective with $T_m^{-1}(x)=\frac{x-x_0}{m}+x_0$ then can easily satisfy that $\mu_n$ converges to $\mu=m.(T_m)^{-1} \#\lambda$. It remains to show that $\mu \in \mc{M}_K(\R^d)$. Using that the convergence in $(\mc{P}_2(\R^d),W_2)$ implies the convergence of the second order moment (see \eqref{eq:W2weakConv}) we have
    \begin{align*}
    M_{x_0}(\mu) &= m\int \vert T_m^{-1}(x) - x_0 \vert^2 d\lambda\\
    &= \frac{1}{m} \int \vert x-x_0 \vert^2 d\lambda\\
    &= \lim_{n} \frac{1}{m}  \int \vert x-x_0 \vert^2 dT_{m_{\mu_n}}\#\bar{\mu}_n\\
    &= \lim_{n} \frac{m_{\mu_n}}{m}  M_{x_0}(\mu_n)\\
    & \le \lim_{n} K \frac{m_{\mu_n}^2}{m}\\
    &= K m.
    \end{align*}
    
    \item Second case : $m=0$ then we want to show that $\mu_n$ converges to $O_\mc{M}$, for this we show that the $W_2$ limit $\lambda=\lim\limits_{n} T_{m_n}\#\bar{\mu}_n$ is equal to $\delta_{x_0}$. Using the triangle inequality we have
    \[
    W_2(\lambda,\delta_{x_0}) \le W_2(\lambda,T_{m_{\mu_n}}\#\bar{\mu}_n ) + W_2(T_{m_{\mu_n}}\#\bar{\mu}_n ,\delta_{x_0}),
    \]
    the term $W_2(\lambda,T_{m_{\mu_n}}\#\bar{\mu}_n )$ converges to $0$. For the second term we use that $\mu_n$ is a sequence of $\mc{M}_K(\R^d)$ and so we have an upper bound
    \begin{align*}
    W_2(T_{m_{\mu_n}}\#\bar{\mu}_n,\delta_{x_0}) & = \int \vert x-x_0 \vert^2 \ud T_{m_{\mu_n}}\#\bar{\mu}_n(x)\\
    & = \frac{1}{m_{\mu_n}}\int \vert T_{m_{\mu_n}}(x)-x_0 \vert^2 \ud\mu_n(x)\\
    &= m_{\mu_n}\int \vert x-x_0 \vert^2 \ud\mu_n(x)\\
    & \le Km_{\mu_n}^2.
   \end{align*}
    Thus $W_2(T_{m_{\mu_n}}\#\bar{\mu}_n,\delta_{x_0})$ converges to 0, which shows that $W_2(\lambda,\delta_{x_0}) = 0$, or equivalently that $\lambda=\delta_{x_0}$, and the Cauchy sequence $\mu_n$ converges  to $0_{\mc{M}}$.
\end{enumerate}
\end{proof}

In order to study the convergence associated with the $\WOP$ metric, let us recall the definition of the weak convergence of measures.
\begin{definition}[Weak convergence of measures]
For a sequence $\mu_n \in \mc{M}(\R^d), n\ge 1$, we say that $\mu_n$ weakly converges to $\mu$ if and only if for all $\phi \in \mc{C}_b(\R^d), $
\[
\int \phi \, \ud\mu_n \to \int \phi \, \ud\mu.
\]
We denote this convergence by $\mu_n \rightharpoonup \mu$.
\end{definition}

The 2-Wasserstein metric metrizes the weak convergence of measures in $\mc{P}_2(\R^d)$ together with the convergence of the second order moments, i.e.
for a sequence $(\mu_n)_{n\ge 1}$ in $\mc{P}_2(\R^d)$ and $\mu \in \mc{P}_2(\R^d)$,
\begin{equation}\label{eq:W2weakConv}
W_2(\mu_n,\mu) \to 0 \iff \left\{ \begin{array}{l} 
M_{x_0}(\mu_n) \to M_{x_0}(\mu)\\
\mu_n \rightharpoonup \mu
\end{array}\right..
\end{equation}
A similar property holds for $\WOP$.
\begin{property}[$\WOP$ metrizes the weak convergence of measures]\label{prop:WOPconv}
Let $\mu_n$ be a sequence in $\mc{M}_K(\R^d)$ and $\mu \in \mc{M}_K(\R^d)$, then 
\[
\WOP(\mu_n,\mu) \to 0 \iff \left\{ \begin{array}{l} 
M_{x_0}(\mu_n) \to M_{x_0}(\mu)\\
\mu_n \rightharpoonup \mu
\end{array}\right.
\]
\end{property}

\begin{proof}
Let $\mu_n$ be a sequence of $\mc{M}_K(\R^d)$ and let $\mu \in \mc{M}_K(\R^d)$. Then using the formulation $(\eqref{eq:defbis})$ we have 

\begin{equation}\label{eq:convergence}
 \WOP(\mu_n,\mu) \to 0 \iff \left\{\begin{array}{l}
        m_{\mu_n} \to m_\mu\\
        (m_{\mu_n}-m_\mu)(M_{x_0}(\mu_n)-M_{x_0}(\mu)) + m_{\mu_n} m_\mu W_2^2(\bar{\mu}_n,\bar{\mu}) \to 0
        \end{array}\right..
\end{equation}
We separate the discussion into two cases:\\
In the first case, we consider $m_\mu>0$. Then the equivalence \eqref{eq:convergence} becomes
    \begin{align*}
        \WOP(\mu_n,\mu) \to 0 & \underset{(a)}{\iff} \left\{\begin{array}{l}
        m_{\mu_n} \to m_\mu\\
        W_2(\bar{\mu}_n,\bar{\mu}) \to 0
        \end{array}\right.\\
        &\underset{(b)}{\iff} \left\{\begin{array}{l}
        m_{\mu_n} \to m_\mu\\
        \bar{\mu}_n  \rightharpoonup \bar{\mu}\\
        M_{x_0}(\bar{\mu}_n) \to M_{x_0}(\bar{\mu})
        \end{array}\right.\\
        &\underset{(c)}{\iff} \left\{\begin{array}{l}
        \mu_n \rightharpoonup \mu\\
        M_{x_0}(\mu_n) \to M_{x_0}(\mu)\\
        \end{array}\right.\\
    \end{align*}
    The equivalence (a) holds since $(m_{\mu_n}-m_\mu)(M_{x_0}(\mu_n)-M_{x_0}(\mu))$ converges to $0$ when $m_{\mu_n}$ converges to $m_\mu$.
    Indeed,
    \[ 
    \vert (m_{\mu_n}-m_\mu)(M_{x_0}(\mu_n)-M_{x_0}(\mu)) \vert \le \vert m_{\mu_n} - m_\mu \vert (M_{x_0}(\mu_n)+M_{x_0}(\mu)) \le \vert m_{\mu_n} - m_\mu \vert (K m_{\mu_n}+M_{x_0}(\mu)) \to 0
    \]
    The equivalence $(b)$ holds since $W_2$ metrizes the weak convergence of measure + convergence of moments.\\
    The equivalence $(c)$ is a direct consequence of the definition of weak convergence.\\
For the second case, consider $m_\mu=0$ and recall that $\mu_n\in\mc{M}_K(\R^d)$. 
Then the equivalence \eqref{eq:convergence} becomes
    \begin{align*}
        \WOP(\mu_n,0_{\mc{M}}) \to 0 & \underset{(d)}{\iff} 
        m_{\mu_n} \to 0\\
        &\underset{(e)}{\iff} \left\{\begin{array}{l}
        \mu_n \rightharpoonup 0_{\mc{M}}\\
        M_{x_0}(\mu_n) \to 0\\
        \end{array}\right.
    \end{align*}    
    The equivalence (d) hold true because $(m_{\mu_n}-m_\mu)(M_{x_0}(\mu_n)-M_{x_0}(\mu))+m_{\mu_n} m_\mu W_2(\bar{\mu}_n,\bar{\mu})$ converges to $0$ as $m_{\mu_n}$ converges to $0$. Indeed
    \[
    \vert (m_{\mu_n}-m_\mu)(M_{x_0}(\mu_n)-M_{x_0}(\mu))+m_{\mu_n} m_\mu W_2(\bar{\mu}_n,\bar{\mu}) \vert =  m_{\mu_n}M_{x_0}(\mu_n)
    \le Km_{\mu_n}^2 \to 0.
    \]
    For the equivalence $(e)$, the implication $(\Leftarrow)$ is direct with the definition of weak convergence.
    For the other implication $(\Rightarrow)$, let $\phi \in \mc{C}_b(\R^d)$ then we have
    \[ 
     \left\vert \int \phi d\mu_n \right\vert \le \Vert \phi \Vert_{\infty}m_{\mu_n} \to 0.
    \]
    Thus $\mu_n \rightharpoonup 0_{\mc{M}}$. And for the second-order moment, observe that
    \[
    M_{x_0}(\mu_n) \le Km_{\mu_n} \to 0,
    \]
which concludes the proof.
\end{proof}

\subsection{Dual formulation}
The $\WOP$ benefits from a dual formulation similar to the $W_2$ dual formulation metric described in \eqref{eq:W2_dual}.
\begin{theorem}[Dual formulation]
The dual formulation of the $\WOP$ metric is given by
\[
\WOP^2(\mu,\nu) = \sup\limits_{(\phi,\psi) \in \Gamma_{\mu,\nu}} \int \phi \ud\mu + \int \psi d\nu,
\]
where $\Gamma_{\mu,\nu}$ is the set of functions
\[
\Gamma_{\mu,\nu}=\left\{ (\phi,\psi) \in \mc{C}_b(\R^d)^2 \bigg\vert m_\mu\phi(x) + m_\nu\psi(y) \le ( m_\mu- m_\nu ) ^2 + \vert T_{m_\mu}(x)-T_{m_\nu}(y) \vert^2,   \mu \otimes \nu \text{-a.e.} \right\}.
\]
\end{theorem}

\begin{proof}
The idea of the proof is to write the $\WOP$ metric as a 2-Wasserstein metric on a higher dimension space, adding on dimension for the mass of the measure and taking  a Dirac measure on that mass.
Then, using the dual formulation (equation~\eqref{eq:W2_dual}), we infer a dual-like formulation of $\WOP$. 
Indeed,
\begin{align*}
\WOP^2(\mu,\nu) &= W_2^2(\delta_{m_\mu} \otimes T_{m_\mu}\#\bar\mu, \delta_{m_\nu} \otimes T_{m_\nu}\#\bar\nu )\\
&= \sup\limits_{ (\phi,\psi) \in \Phi}  \int \phi(\alpha,x) \ud \delta_{m_\mu}(\alpha)\otimes T_{m_\mu}\#\bar\mu(x) + \int \phi(\beta,y) \ud \delta_{m_\nu}(\beta)\otimes T_{m_\nu}\#\bar\nu(y)\\
&= \sup\limits_{ (\phi,\psi) \in \Phi}  \int \frac{\phi(m_\mu,T_{m_\mu}(x))}{m_\mu} \ud\mu(x) + \int \frac{\phi(m_\nu,T_{m_\nu}(y))}{m_\nu}\ud \nu(y).
\end{align*}
Where the space $\Phi$ is the set of functions $(\phi,\psi) \in \mc{C}_b(\R_+ \times \R^d)^2$ such that for all $(\alpha,x,\beta,y)$ defined $\delta_{m_\mu} \otimes T_{m_\mu}\#\bar\mu \otimes \delta_{m_\nu} \otimes T_{m_\nu}\#\bar\nu$-almost everywhere we have the inequality
\[
\phi(\alpha,x) + \psi(\beta,y) \le (\alpha-\beta)^2 + \vert x-y \vert^2, 
\]
or equivalently for all $(x,y)$ defined $\mu \otimes \nu$-almost everywhere, we have the inequality 
\[
\phi(m_\mu,T_{m_\mu}(x)) + \psi(m_\nu,T_{m_\nu}(y)) \le (m_\mu-m_\nu)^2 + \vert T_{m_\mu}(x)-T_{m_\nu}(y) \vert^2.
\]
Setting the functions $\tilde{\phi}(x)= \frac{\phi(m_\mu,T_{m_\mu}(x))}{m_\mu}$ and $\tilde{\psi}(y)=\frac{\phi(m_\nu,T_{m_\nu}(y))}{m_\nu}$, we finally obtain
\[ 
\WOP^2(\mu,\nu)=\sup \int \tilde{\phi}(x)\ud\mu(x) + \int \tilde{\psi}(y)d\nu(y),
\]
where the functions $(\tilde{\phi},\tilde{\psi}) \in \mc{C}_b(\R^d)^2$ satisfy
\[ 
m_\mu \tilde{\phi}(x) + m_\nu \tilde{\psi}(y) \le (m_\mu- m_\nu)^2 + \vert T_{m_\mu}(x) - T_{m_\nu}(y) \vert^2,
\] 
for $(x,y)$ defined $\mu \otimes \nu$-almost everywhere.
\end{proof}

\subsection{Geodesics of $\WOP$}
We recall that a metric space $(D,\mc{X})$ is geodesic if for every pair of point $(x,y)$ in $\mc{X}^2$ there exist a path $\sigma_t:[0,1] \longrightarrow \mc{X}$, called a \emph{(constant speed) geodesic}, starting at $x$, ending at $y$, such that
\[ 
D(\sigma_t,\sigma_s) = \vert t-s \vert D(x,y), \quad \forall t,s \in [0,1].
\]
In the case of $\big( W_2,\mc{P}_2(\R^d) \big)$ the geodesic are well known and can be expressed in terms of the Monge map between the two end points.
The following result provides a description of the geodesics of $\WOP$ and link them to the $W_2$ geodesics between the normalized measures.

\begin{property}[Geodesics for $\WOP$]\label{prop:geodesics}
Given $\mu_0,\mu_1 \in \mc{M}(\R^d)$, the geodesic is 
\begin{equation}\label{eq:Geod}
\mu_t =m_t.\hat\mu_{\lambda_t}, \quad t \in [0,1],
\end{equation}
where the mass is $m_t=(1-t)m_{\mu_0}+t m_{\mu_1}$, the probability measure $\hat\mu_t$ is the geodesic in $(\mc{P}_2(\R^d),W_2)$ between $\bar\mu_0$ and $\bar\mu_1$, and the parametrization is $\lambda_t = \frac{tm_{\mu_1}}{(1-t)m_{\mu_0}+tm_{\mu_1}}$.
In particular, the geodesics are independent on the reference point $x_0$.

Moreover, the space $(\mc{M}(\R^d),\WOP)$ is a geodesic space with positive curvature in the sense of Alexandrov. 
\end{property}

\begin{proof}
Let $\mu_0,\mu_1 \in \mc{M}(\R^d)$, using the fact that $(\R,\vert\,.\,\vert)$ and $(\mc{P}_2(\R^d),W_2)$ are geodesic spaces, using the definition of $\WOP$ we deduce that the geodesic in $(\mc{M}(\R^d),\WOP)$ between $\mu_0$ and $\mu_1$ is
\[
\mu_t=m_t.T_{m_t}^{-1}\#\eta_t,
\]
where $m_t=(1-t)m_{\mu_0} + t m_{\mu_1}$ and $\eta_t$ is the geodesic between $T_{m_{\mu_0}}\#\bar{\mu}_0$ and $T_{m_{\mu_1}}\#\bar{\mu}_1$ in $(\mc{P}_2(\R^d),W_2)$.\\
Moreover in $(\mc{P}_2(\R^d),W_2)$ we can describe the geodesic more in detail 
\[
\eta_t=P_t\#\pi, 
\]
with $P_t(x,y)=(1-t)x+ty$ and with $\pi \in \Pi_o(T_{m_{\mu_0}}\#\bar{\mu}_0,T_{m_{\mu_1}}\#\bar{\mu}_1)$.\\
Therefore,
\[
\pi = (T_{m_{\mu_0}},T_{m_{\mu_1}})\#\tilde{\pi},
\]
with $\tilde{\pi} \in \Pi_o(\bar{\mu}_0,\bar{\mu}_1)$.
Thus the geodesic can be rewritten as follow
\begin{align*}
\mu_t & = m_t.T_{m_t}^{-1}\#P_t\#(T_{m_{\mu_0}},T_{m_{\mu_1}})\#\tilde{\pi}\\
& = m_t.\big(T_{m_t}^{-1} \circ P_t \circ (T_{m_{\mu_0}},T_{m_{\mu_1}}) \big)\#\tilde{\pi}.
\end{align*}
Setting $\lambda_t = \frac{t m_{\mu_1}}{(1-t)m_{\mu_0} + t m_{\mu_1}}$, we can simplify $\big(T_{m_t}^{-1} \circ P_t \circ (T_{m_{\mu_0}},T_{m_{\mu_1}}) \big)$ :
\begin{align*}
\big(T_{m_t}^{-1} \circ P_t \circ (T_{m_{\mu_0}},T_{m_{\mu_1}}) \big)(x,y)  &= \frac{(1-t) \big[ m_{\mu_0}(x-x_0) + x_0 \big]  + t \big[ m_{\mu_1}(y-x_0)+ x_0 \big]  - x_0 }{(1-t)m_{\mu_0} + tm_{\mu_1}} + x_0\\
& = \frac{(1-t) m_{\mu_0} x + t m_{\mu_1} y}{(1-t)m_{\mu_0} + t m_{\mu_1}}\\
& = (1-\lambda_t)x + \lambda_t y\\
&= P_{\lambda_t}(x,y)
\end{align*}
Using that $P_t\#\tilde{\pi}$ is a geodesic between $\bar{\mu}_0$ and $\bar{\mu}_1$ which we denote $\hat{\mu}_t$, mean we can interpret the geodesic as a time reparametrization
\[
\mu_t  = m_t.P_{\lambda_t}\#\tilde{\pi}=m_t.\hat{\mu}_{\lambda_t}.
\]
The immediate advantage of having rewritten the geodesic as such is that there is no dependence on the reference point $x_0$.
Finally, the positive curvature of $(\mc{M}(\R^d),\WOP)$ is a direct consequence of the $0$ curvature of $(\R^d,\vert \, . \, \vert)$ and the positive curvature of $(\mc{P}_2(\R^d),W_2)$.
Indeed for all $t \in [0,1]$, for all $\mu_0,\mu_1 \in \mc{M}(\R^d)$, there exists $\mu_t$ which is the geodesic between $\mu_0$ and $\mu_1$ defined previously, such that for all $\eta \in \mc{M}(\R^d)$ we have :

\begin{align*}
\WOP^2(\mu_t,\eta) = & \vert m_{t}-m_\eta \vert^2 + W_2^2 (T_{m_{t}}\#\bar{\mu}_t,T_{m_\eta}\#\bar{\eta})\\
= & (1-t)\vert m_{\mu_0} - m_\eta \vert^2 + t\vert m_{\mu_1} - m_\eta \vert^2 -t(1-t)\vert m_{\mu_0} - m_{\mu_1} \vert^2 + W_2^2 (T_{m_{t}}\#\bar{\mu}_t,T_{m_\eta}\#\bar{\eta})\\
\ge &(1-t) \bigg[ \vert m_{\mu_0}- m_\eta \vert^2 + W_2^2(T_{m_{\mu_0}}\#\bar{\mu}_0,T_{m_\eta}\#\bar{\eta}) \bigg]  + t \bigg[ \vert m_{\mu_1} - m_\eta \vert^2 + W_2^2(T_{m_{\mu_1}}\#\bar{\mu}_1,T_{m_\eta}\#\bar{\eta}) \bigg]\\
& -t(1-t) \bigg[ \vert m_{\mu_0} - m_{\mu_1} \vert^2 + W_2^2 (T_{m_{\mu_0}}\#\bar{\mu}_0,T_{m_{\mu_1}}\#\bar{\mu}_1) \bigg] \\
= &(1-t)\WOP^2(\mu_0,\eta) + t\WOP^2(\mu_1,\eta) - t(1-t)\WOP^2(\mu_0,\mu_1).
\end{align*}
\end{proof}

\section{First order calculus on $\WOP$}\label{sec:Dynamic}

In this section, we first describe the dynamic interpretation of $\WOP$.
In a second time we describe its Riemannian-like interpretation: the tangent space as well as the gradient for a broad set of functional on $\mc{M}(\R^d)$ and the gradient-flow equation.
Finally, we give a direct way to extend functional from $\mc{P}_2(\R^d)$ into $\mc{M}(\R^d)$ while conserving a mass conservative gradient-flow. 

\subsection{Dynamic formulation}
Let us start with deriving an formal computation for a dynamic formulation.
Recall from~\eqref{eq:Geod} that the geodesic between $\mu_0$ and $\mu_1$ is given by
\[
\mu_t=m_t.\hat\mu_{\lambda_t},
\]
where $\lambda_t=\frac{tm_1}{(1-t)m_0+tm_1}$, $\hat\mu_t$ is the constant speed geodesic between $\bar\mu_0$ and $\bar\mu_1$, and $m_t=(1-t)m_0+tm_1$.
Then, for $v=\nabla\phi-{\rm Id}$ the tangent vector in $W_2$ pointing from $\bar\mu_0$ to $\bar\mu_1$,
\[
    \partial_t\mu_t=\partial_t m_t \cdot \hat\mu_{\lambda_t}+m_t\cdot\partial_t\lambda_t\cdot\partial_s\hat\mu_s|_{s=\lambda_t}.
    \]
    And thus,
\begin{align*}
 \partial_t\mu_t|_{t=0}&=(m_1-m_0)\bar\mu_0 - m_0\cdot\partial_{t|t=0}\lambda_t\cdot\dive (v\bar\mu_0)\\
    &=\left(\frac{m_1}{m_0}-1\right)\mu_0 - \frac{m_1}{m_0}\dive \left(v\mu_0\right)\\
    &=\frac{m'_0}{m_0}\mu_0-\dive \left(v\left(1+\frac{m'_0}{m_0}\right)\mu_0\right).
\end{align*}
In this equation, the term $m':=m'_0=m_1-m_0$ encodes the amount of mass created locally, and $u:=v\bigg(1+\frac{m'}{m_0}\bigg)$ is a vector field displacing $\bar\mu_0$.
These two quantities describes the direction of displacement from $\mu_0$ toward $\mu_1$ and can be thought as elements of tangent plane.
Since the distance between $\mu_1$ and $\mu_0$ in the tangent plane at $\mu_0$ is the same as their true distance, we expect that the distance to be defined in the tangent plane $\|\cdot\|_{\mu_0}$ at $\mu_0$ to satisfy
\[
\Vert(m',u)\Vert_{\mu_0}^2 =\WOP^2(\mu_0,\mu_1)=(m_1-m_0)^2 + \int \vert y-x \vert^2 \ud \pi(x,y),
\]
for some $\pi \in \Pi_o(T_{m_0}\#\bar{\mu}_0, T_{m_1}\#\bar{\mu}_1)$. 
By Brenier's theorem, we know that $\pi=(T_{m_0},T_{m_1})\#\big( \bar{\mu}_0\otimes \delta_{\{y=\nabla \phi(x) \} } \big)$, and thus
\begin{align*}
\Vert(m',u)\Vert_{\mu_0}^2 &= m'^2 + \int \vert T_{m_1}\big(\nabla \phi(x) \big) - T_{m_0}(x) \vert^2 d\bar{\mu}_0(x)\\
&= m'^2 + \int \vert m_1\big( \nabla\phi(x) -x_0 \big) - m_0(x-x_0) \vert^2 d\bar{\mu}_0(x)\\
&=m'^2 + \int \vert m_1 v + (m_1 - m_0)(x-x_0) \vert^2 d\bar{\mu}_0(x)\\
&=m'^2 + \int \vert m_0 u + m'(x-x_0) \vert^2 d\bar{\mu}_0(x).
\end{align*}
In other words, the tangent plane ought to consist of two elements: a real number $m'$ to account for the change of mass, and a vector field $u$ to account for the transport of the normalized mass; and the $\WOP$ metric integrate these two elements according to the formula above.
We discuss with more details tangent space in the Section~\ref{subsec:GF}.
These observations readily suggest the following dynamic formulation.

\begin{theorem}[Dynamic formulation]
The $\WOP$ metric satisfies that for all $\mu,\nu\in\mc{M}(\R^d)$, 
\[
\WOP^2(\mu,\nu)=\inf \bigg\{ \int |m'_t|^2+|m_t u_t + m'_t(x-x_0) |^2 \ud\bar\mu_t(x)\ud t \bigg\}
\]
where the infimum is taken among the triplet $(u_t,m_t,\mu_t)$ such that $m_t \in \mc{C}^1([0,1])$, $u_t:[0,1] \to L^2(\mu_t)$, $\mu_t : [0,1] \to \mc{M}(\R^d)$, and
\[
\begin{cases}
\partial_t \mu_t+ \dive \left( u_t \mu_t \right) =\frac{m'_t}{m_t}\mu_t  \label{eq:TransportSource}\\
    m_0=m_\mu, \quad  m_1= m_{\nu}, \nonumber \\
     \mu_0=\mu,\quad \mu_1=\nu.\nonumber
\end{cases}
\]
This equation is to be understood weakly, i.e to be verified against the set of test functions $\mc{C}^{\infty}(\R^d \times [0,1])$ with compact support.
\end{theorem}

\begin{proof}
Using the $W_2$ dynamic formulation in the definition of $\WOP$, we obtain
\begin{align*}
\WOP^2(\mu_0,\mu_1)&=(m_0-m_1)^2+W_2^2(T_{m_0}\#\bar\mu_0, T_{m_1}\#\bar\mu_1)\\
&=\inf\left\{ \int |m'_t|^2+\int|v_t|^2\ud\nu_t\ud t:\,\nabla m_t'=0, \partial_t \nu_t+\dive (v_t\nu_t)=0, \nu_i=T_{m_i}\#\bar\mu_i,i=1,2\right\}\\
&=\inf\left\{ \int |m'_t|^2+\int|v_t\circ T_{m_t}|^2\ud\bar\mu_t\ud t:\,\nabla m_t'=0, \partial_t \nu_t+\dive (v_t\nu_t)=0, \nu_i=T_{m_i}\#\mu_i,i=1,2\right\}
\end{align*}
For each $\nu_t$, we set $\mu_t=m_t\cdot T^{-1}_{m_t}\#\nu_t=m_t^{d+1}\cdot \nu_t\circ T_{m_t}$.
Thus, $\nabla\mu_t=m_t^{d+2}\nabla\nu_t\circ T_{m_t}$ and
\begin{align*}
    \partial_t\mu_t&=(d+1)m_t^dm'_t \nu_t\circ T_{m_t}+ m_t^{d+1}\partial_t(\nu_t\circ T_{m_t})\\
    &=\frac{(d+1)}{m_t}m'_t\mu_t+m_t^{d+1}\left[(\partial_t\nu_t)\circ T_{m_t}+ \langle\partial_t T_{m_t},\nabla\nu_t\circ T_{m_t}\rangle \right]\\
    &=\frac{(d+1)}{m_t}m'_t\mu_t+m_t^{d+1}(\partial_t\nu_t)\circ T_{m_t}+ m_t^{-1} \langle \partial_t T_{m_t},\nabla\mu_t\rangle \\
\end{align*}
On the other hand,
\begin{align*}
    (\partial_t\nu_t)\circ T_{m_t}&=-\dive (v_t\nu_t)\circ T_{m_t}\\
    &=-\dive (v_t)\circ T_{m_t}\cdot\nu_t\circ T_{m_t}-\langle v_t\circ T_{m_t},\nabla\nu_t\circ T_{m_t}\rangle\\
    &=-m_t^{-(d+1)}\dive (v_t)\circ T_{m_t}\cdot\mu_t - m_t^{-(d+2)}\langle v_t\circ T_{m_t},\nabla\mu_t\rangle\\
    &=-m_t^{-(d+2)}\dive (v_t\circ T_{m_t})\cdot\mu_t - m_t^{-(d+2)}\langle v_t\circ T_{m_t},\nabla\mu_t\rangle
\end{align*}
so we get, setting $m_t^{-1} v_t\circ T_{m_t}=\tilde{u}_t$
\begin{align*}
\partial_t\mu_t&=\frac{(d+1)}{m_t}m'_t\mu_t -m_t^{-1}\dive (v_t\circ T_{m_t})\cdot\mu_t-m_t^{-1}\langle v_t\circ T_{m_t},\nabla\mu_t\rangle + m_t^{-1} \langle \partial_t T_{m_t},\nabla\mu_t\rangle \\
&=\frac{(d+1)}{m_t}m'_t\mu_t -\dive ( \tilde{u}_t)\cdot\mu_t -\langle \tilde{u}_t,\nabla\mu_t \rangle+ m_t^{-1} \langle \partial_t T_{m_t},\nabla\mu_t\rangle \\
&=\frac{(d+1)}{m_t}m'_t\mu_t -\dive (\tilde{u}_t\mu_t)+ m_t^{-1} \langle \partial_t T_{m_t},\nabla\mu_t\rangle \\
&=\frac{m'_t}{m_t}\mu_t -\dive (\tilde{u}_t\mu_t) + m_t^{-1}\dive (\partial_t T_{m_t}\mu_t)\\
&=\frac{m'_t}{m_t}\mu_t -\dive \left( \big( \tilde{u}_t-m_t^{-1}\partial_t T_{m_t} \big)\mu_t \right)\\
&=\frac{m'_t}{m_t}\mu_t -\dive \left( \big(  \tilde{u}_t-\frac{m'_t}{m_t} ( x-x_0 ) \mu_t \big)\right)\\
\end{align*}
Therefore
\[
\WOP^2(\mu_0,\mu_1)=\inf\left\{ \int |m'_t|^2+\int |m_tu_t|^2 \ud\bar{\mu}_t\ud t:\,\nabla m_t'=0, \partial_t \mu_t + \dive \left( \big( \tilde{u}_t-\frac{m'_t}{m_t}(x-x_0) \big)\mu_t \right) =\frac{m'_t}{m_t}\mu_t  \right\}.
\]
Setting $u_t=\tilde{u}_t-\frac{m'_t}{m_t} ( x-x_0 )$, we get the result.
\end{proof}

\subsection{Gradient flows}\label{subsec:GF}

As mentioned in the previous section, the space $\big(\mc{P}_2(\R^d),W_2 \big)$ has a Riemannian-like structure where the tangent space $\mathbf{T_{\mu} \mc{P}_2(\R^d)}$ at each measure $\mu$ is a subset of vector fields in $L^2(\mu)$. 
More precisely, it is defined by
\[    
\mathbf{T_\mu \mc{P}_2(\R^d)}=\overline{\left\{ \nabla \phi \bigg\vert \phi \in \mc{C}^{\infty}_c(\R^d) \right\}}^{L^2(\mu)}.
\]
We refer the reader to \cite[Chapter 8.4]{Ambrosio2005} for more details.

Inspired by the observation of the previous section, and the dynamic formulation, the space $\big(\mc{M}(\R^d),\WOP \big)$ can also be endowed with a Riemannian-like structure.

\begin{definition}[Tangent space]
    We define the tangent space of $\mc{M}(\R^d)$ at $\mu$ by
    \[
    \mathbf{T_\mu \mc{M}(\R^d)} := \mathbf{T_{\bar{\mu}} \mc{P}_2(\R^d)} \times \R.
    \]
    For a tangent vector $(u,m') \in \mathbf{T_\mu \mc{M}(\R^d)}$, there corresponds a geodesic $(\mu_t)_{t\in[0,1]}$ emanating from $\mu_0=\mu$ such that
\begin{equation}\label{eq:taneq}
\partial_t \mu_t \vert_{t=0} = \frac{m'}{m_\mu}\mu - \dive(u \mu).
\end{equation}
We endow $\mathbf{T_\mu \mc{M}(\R^d)}$ with the scalar product defined by
\begin{equation}\label{eq:ScalarProd}
\langle (u_1,m'_1),(u_2,m'_2) \rangle_\mu = m'_1 m'_2 +\int \langle m_\mu u_1(x) +m'_1 (x-x_0), m_\mu u_2(x) + m'_2(x-x_0) \rangle d\bar{\mu}(x).
\end{equation}
\end{definition}

Using the tangent space, we can mimic the chain rule formula on Riemannian manifold to define a gradient on the $\WOP$ space and a gradient flow of a functional.

Recall that on a Riemannian manifold $(\mc{M},g)$ the gradient flow of a smooth functional $F:\mc{X}\longrightarrow\R$ is a curve $(x_t)_{t\ge 0}$ following the direction from the negative gradient of $F$ at each point.
It is thus characterized by the equation
\[ 
\partial_t x_t = -\nabla F( x_t ),\quad \forall t \ge 0.
\] 
\begin{definition}[$\WOP$ gradient]
    For a functional $F:\mc{M}(\R^d)\to\R$, the $\WOP$ gradient of $F$ at a point $\mu$ is defined as the $\nabla_\WOP F(\mu)\in\mathbf{T_\mu \mc{M}(\R^d)}$ such that for all locally Lipchitz path $(\mu_t)_{t\in[0,1]}$ such that $\mu_0=\mu$ and $\partial_t \mu_t \vert_{t=0} = \frac{m'}{m_\mu}\mu - \dive(u \mu)$,
    \begin{equation}
    \frac{\ud }{\ud t}F(\mu_t)|_{t=0}=\left\langle\nabla_\WOP F(\mu),(m',u)\right\rangle_\mu.\label{eq:WOPchainrule}
    \end{equation}
    A continuous path $(\mu_t)_{t\ge0}$ valued in $\mc{M}(\R^d)$ is a gradient flow of $F$ if $\nabla_\WOP(F)(\mu_t)=(m'_t,u_t)$ and $(\mu_t)_{t\ge0}$ is a weak solution of
    \[
    \partial_t\mu_t=\frac{m'_t}{m_{\mu_t}}\mu_t+\dive(u_t\mu_t).
    \]
\end{definition}

In order to characterize gradient flows on the $\WOP$ space, let us consider functionals $F:\mc{M}(\R^d) \to \R$ such that for $\mu\in\mc{M}(\R^d)$, their exists a function a function $\frac{\delta F}{\delta \mu} (\mu)\in \mc{C}^1(\R^d)$ called the \emph{first variation} of $F$ at $\mu$ such that for all $\nu \in \mc{M}(\R^d)$ 
\begin{equation}\label{eq:WOPfirstVariation}
    F(\nu) = F(\mu) +  \int \frac{\delta F}{\delta \mu}(\mu) \ud (\nu-\mu) + o\big(\WOP(\mu, \nu)\big).
\end{equation}
\begin{theorem}[$\WOP$-gradient formula]\label{thm:Wgradient}
Let $F$ be a functional on $\mc{M}(\R^d)$ satisfying~\eqref{eq:WOPfirstVariation} and denote by $\frac{\delta F}{\delta \mu}(\mu)$ its first variation.
Suppose $\|\nabla\frac{\delta F}{\delta\mu}(\mu)(x)\|\lesssim 1+|x|$.
Then the $\WOP$ gradient of $F$, $\nabla_{\WOP}F(\mu) = (u_F(\mu),m'_F(\mu))$ satisfies
\begin{align}\label{eq:WOPgrad}
& m'_{F}(\mu) = \int \frac{\delta F}{\delta \mu}(\mu) \ud\bar{\mu} -  \int \langle \nabla\frac{\delta F}{\delta \mu}(\mu) , (x-x_0) \rangle \ud\bar{\mu},\\
& u_{F}(\mu) = \frac{1}{m_\mu}\left[ \nabla\frac{\delta F}{\delta \mu}(\mu)- m'_{F}(\mu)(x-x_0)\right]. 
\end{align}
In particular, the gradient flow $(\nu_t)_{t\ge 0}$ of $F$ satisfies the weak partial differential equation - i.e. against test functions in $\mc{C}_c^{\infty}(\R^d\times \R^\star_+)$,
\begin{equation}\label{eq:FG}
\partial_t \nu_t = -\frac{m'_F(\nu_t)}{m_{\nu_t}}\nu_t+\dive \big( u_F(\nu_t)\nu_t \big).
\end{equation}
\end{theorem}

\begin{proof}
The gradient of $F$ at $\mu$ is defined as the quantity  $\nabla_{\WOP} F(\mu)\in T_\mu\mc{P}_2(\R^d)$ such that the chain rules holds
\[
\frac{\ud}{\ud t} F(\mu_t) \vert_{t=0}, = \langle \nabla_{\WOP} F(\mu), (u,m') \rangle_{\mu},
\]
where $\mu_t$ is the geodesic at time $t$, starting from $\mu$ and going along the tangent vector represented by $(u,m')$, i.e. satisfying \eqref{eq:taneq}.
In fact, thanks to the assumption \eqref{eq:WOPfirstVariation}, it is enough to check \eqref{eq:WOPchainrule} for geodesics.
Using again the assumption \eqref{eq:WOPfirstVariation} on the first variation, we have
\begin{align*}
    \frac{\ud}{\ud t} F(\mu_t)  &=  \lim_{t\to 0}\frac{F(\mu_t)-F(\mu)}{t}\\
    &= \lim_{t\to 0}\frac{1}{t}\int \frac{\delta F}{\delta \mu}(\mu) \ud( \mu_t - \mu))  \\
    &= \lim_{t\to 0}\frac{m_{\mu_t}-m_\mu}{t} \int \frac{\delta F}{\delta \mu} (\mu)\ud\bar{\mu} + m_{\mu_t} \int \frac{\delta F}{\delta \mu}(\mu)\frac{1}{t}\ud(\bar{\mu}_t-\bar{\mu}).
\end{align*}
The second inequality hold using the assumption \eqref{eq:WOPfirstVariation} on the first variation and the fact that $\WOP(\mu,\mu_t)/t\to0$ as $t\to 0$ for geodesics.
For $t$ going to $0$, the term $\frac{m_{\mu_t}-m_\mu}{t}$ converges to $m'$. 
For the second term, $m_{\mu_t}$ converges to $m_\mu$ as a consequence of the topology of $\WOP$ (Property~\ref{prop:WOPconv}).
Then, observe that
\[ 
\lim\limits_{t \to 0}  \int \frac{\delta F}{\delta \mu}(\mu)\frac{1}{t}\ud(\bar{\mu}_t-\bar{\mu}) = \frac{\ud}{\ud t}\int  \frac{\delta F}{\delta \mu}(\mu)(x+tu) \ud\bar{\mu}(x) \bigg\vert_{t=0}
\]
Using the hypothesis that $\nabla \frac{\delta F}{\delta \mu}(\mu)$ has a linear growth  $\vert \nabla \frac{\delta F}{\delta \mu}(\mu) (x) \vert \lesssim 1+ \vert x \vert$, we can bound the following derivative in a neighborhood $[0,\eps[$ of $t=0$,
\[ 
\bigg\vert \frac{\partial}{\partial t} \bigg( \frac{\delta F}{\delta \mu}(\mu)(x+tu) \bigg) \bigg\vert = \big\vert \langle \nabla \frac{\delta F}{\delta \mu}(x+tu), u \rangle \big\vert \le 2\big\vert \nabla \frac{\delta F}{\delta \mu}(\mu)(x+tu)   \big\vert^2 +2 \vert u \vert^2  \lesssim 1 + \vert x \vert^2  + (1+\eps)\vert u \vert^2
\]
Since $u$ is in $L^2(\bar{\mu})$ and $\bar{\mu}$ admits a second-order moment, the Leibniz integral rule ensures
\[
\lim\limits_{t \to 0}  \int \frac{\delta F}{\delta \mu}(\mu)\frac{1}{t}\ud(\bar{\mu}_t-\bar{\mu}) =  \int \langle \nabla\frac{\delta F}{\delta \mu}(x+tu),u \rangle \ud \bar{\mu}\bigg\vert_{t=0} = \int \langle \nabla\frac{\delta F}{\delta \mu},u \rangle \ud \bar{\mu}.
\]
 To sum up,
 \begin{align*}
 \frac{\ud}{\ud t} F(\mu_t) \vert_{t=0} 
 &= m' \int \frac{\delta F}{\delta \mu} \ud \bar{\mu} + m_\mu\int \langle \nabla \frac{\delta F}{\delta \mu}(x),u\rangle \ud \bar{\mu}\\
 &= m'\bigg[  \int \frac{\delta F}{\delta \mu} \ud \bar{\mu} - \int \langle \nabla \frac{\delta F}{\delta \mu},(x-x_0) \rangle \ud\bar{\mu} \bigg] + \int \langle m_{\mu} u + m'(x-x_0), \nabla \frac{\delta F}{\delta \mu} \rangle \ud \bar{\mu}. 
 \end{align*}
Using the scalar product~\eqref{eq:ScalarProd}, we identify the variation of mass of the $\WOP$ gradient to be
\[
m'_{F}(\mu) = \int \frac{\delta F}{\delta \mu} \ud\bar{\mu} -  \int \langle \nabla\frac{\delta F}{\delta \mu} , (x-x_0) \rangle \ud\bar{\mu}, 
\]
and the vector field $u_{F}(\mu)$ to be defined by
\[
\nabla\frac{\delta F}{\delta \mu}(\mu)(x)=m_\mu u_{F}(\mu)(x) +m'_{F}(\mu)(x-x_0)
\]
i.e.
\[
 u_{F}(\mu)(x) = \frac{1}{m_\mu}\nabla\frac{\delta F}{\delta \mu}(\mu)(x)-\frac{m'_{F}(\mu)}{m_\mu}(x-x_0).
\]
\end{proof}

We now use the same visual interpretation of the $\WOP$ metric as in Figure \ref{fig:my_label}, and describe an intuitive characterization of the functionals for which the gradient flow satisfies a conservative partial differential equation, i.e for which the total mass $m_{\nu_t}$ is constant.
The Figure \ref{fig:my_label2} hints that the gradient of a functional $F$ must be \emph{orthogonal to the half-lines} $\{ aT_{\frac{1}{a}}\#\bar\mu \vert a > 0 \}$ for every measures $\mu$. This is equivalent to $F$ satisfying the mass invariant equation
\begin{equation}\label{eq:massInvariance}
F(\mu)=F(aT_{\frac{1}{a}\#{\mu}}), \quad \forall a>0, \; \forall \mu \in \mc{M}(\R^d).
\end{equation}

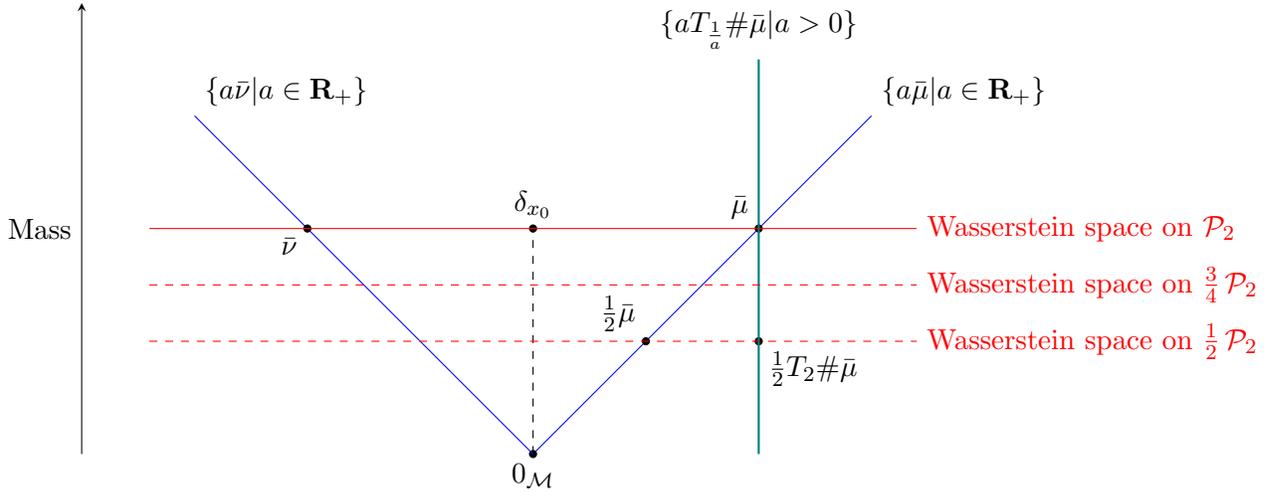
\begin{figure}[h]
    \centering
\begin{tikzpicture}[scale=3]
\draw node[circle,fill,scale=0.3] (a) at (0,1) {};
\draw node[anchor=south] (b) at (0,1) {$\delta_{x_0}$};
\draw[dashed] (0,0) -- (0,1);

\draw[blue] (1.5,1.5) -- (0,0);
\draw node[circle,fill,scale=0.3] () at (0,0) {};
\draw node[anchor=north] () at (0,0) {$0_{\mc M}$};

\draw node[circle,fill,scale=0.3] (a) at (1,1) {};
\draw node[anchor=south east] (b) at (1,1) {$\bar\mu$};
\draw node (a) at (1.5,1.5) {};
\draw node[anchor=south west] (b) at (1.5,1.5) {$\{a\bar\mu \vert a \in \R_+ \}$}; 
\draw node[circle,fill,scale=0.3] (c) at (1,0.5) {};
\draw node[anchor=north west] (c) at (1,0.5) {$\frac{1}{2}T_2\#\bar\mu$};
\draw node[circle,fill,scale=0.3] (d) at (0.5,0.5) {};
\draw node[anchor=south east] (d) at (0.5,0.5) {$\frac{1}{2}\bar\mu$};
 
\draw[thick,teal] (1,0) -- (1,1.75);
\draw node (a) at (1,1.75) {};
\draw node [anchor=south] (b) at (1,1.75) {$\{ aT_{\frac{1}{a}}\#\bar \mu \vert a > 0 \}$};

\draw[blue] (-1.5,1.5) -- (0,0);
\draw node[circle,fill,scale=0.3] (a) at (-1,1) {};
\draw node[anchor=north east] (b) at (-1,1) {$\bar\nu$};
\draw node (c) at (-1.5,1.5) {};
\draw node [anchor=south west] (d) at (-1.5,1.5) {$\{ a\bar \nu \vert a \in \R_+ \}$};

\draw[red] (-1.7,1) -- (1.7,1) node[anchor=west]{Wasserstein space on $\mc P_2$};
\draw[red,dashed] (-1.7,0.75) -- (1.7,0.75) node[anchor=west]{Wasserstein space on $\frac{3}{4}\,\mc P_2$};
\draw[red,dashed] (-1.7,0.5) -- (1.7,0.5) node[anchor=west]{Wasserstein space on $\frac{1}{2}\,\mc P_2$};

\draw[-stealth] (-2,0) -- (-2,2);
\draw node[anchor=east] () at (-2,1) {Mass};

\end{tikzpicture}
    \caption{Using the analogy of figure \ref{fig:my_label}, we highlight the intuition of some orthogonality between the planes with constant mass (i.e proportional to $\mc{P}_2$), and the half-line $\{aT_{\frac{1}{a}}\#\bar \mu \vert a >0 \}$. }
    \label{fig:my_label2}
\end{figure}

In fact, the following result holds.
\begin{theorem}[Characterization of conservative functionals]
Let $F$ be a functional on $\mc{M}(\R^d)$ satisfying~\eqref{eq:WOPfirstVariation}.
Recall that the variation of the total mass of the $\WOP$ gradient of $F$ at $\mu$ is given by $m'_F(\mu)=\int \frac{\delta F}{\delta \mu} \ud \bar{\mu} - \int \langle \nabla \frac{\delta F}{\delta \mu}, (x-x_0) \rangle \ud \bar{\mu}$.
Denote by $\nu_t$ be a gradient flow of $F$ (i.e. satisfying equation~\eqref{eq:FG}) starting at $\nu_0$ then the following assertions are equivalent.
\begin{enumerate}
    \item $F(\mu) = F(aT_{\frac{1}{a}}\#\mu),  \quad \forall \mu \in \mc{M}(\R^d)\setminus\{O_\mc{M}\}, \: \forall a>0$
    \item $m_{F}'(\mu) = 0, \quad \forall \mu \in \mc{M}(\R^d)\setminus\{O_\mc{M}\}$
    \item $m_{\nu_t}$ is constant for all $t >0$ and all $\nu_0$.
\end{enumerate}
\end{theorem}

\begin{proof}
Set $\mu_a= aT_{\frac{1}{a}}\#\mu$.
We first prove the equivalence between assertions $1.$ and $2.$.
The item $1.$ is equivalent to $\frac{\ud}{\ud a}F(\mu_a)=0$ and using the first variation, for $\eps > -a$,
\begin{align*}
    \frac{F(\mu_{a+\eps}) - F(\mu_a)}{\eps} &= \underset{(A)}{\underbrace{\frac{1}{\eps} \int \frac{\delta F}{\delta \mu_a} \ud( \mu_{a+\eps} - \mu_a )}} + o \left( \underset{(B)}{\underbrace{\frac{1}{\eps} \WOP(\mu_{a+\eps},\mu_a)}} \right)
\end{align*}
We first compute the term $(B)$
\[
    \frac{\WOP(\mu_{a+\eps},\mu)}{\eps} = \sqrt{ \frac{ ((a+\eps) - a)^2 }{\eps^2} + \frac{W_2^2(\bar{\mu},\bar{\mu})}{\eps^2} }= 1
\]
Now for the term $(A)$, using similar arguments to the proof of Theorem~\ref{thm:Wgradient},
\begin{align*}
    (A) =& \bigg( \frac{ (a+\eps)-a}{\eps} \bigg) \int \frac{\delta F}{\delta \mu_a} \ud T_{\frac{1}{a}}\#\bar{\mu} + (a+\eps) \int \frac{\delta F}{\delta \mu_a} \frac{1}{\eps}\ud( T_{\frac{1}{a+\eps}}\#\bar{\mu} - T_{\frac{1}{a}}\#\bar{\mu})\\
    =& \int \frac{\delta F}{\delta \mu_a} \ud T_{\frac{1}{a}}\#\bar{\mu}  + (a+\eps) \int \frac{1}{\eps} \bigg( \frac{\delta F}{\delta\mu_a}\big( T_{\frac{1}{a+\eps}} \big) - \frac{\delta F}{\delta\mu_a}\big( T_{\frac{1}{a}} \big) \bigg) \ud \bar{\mu}\\
    \underset{\eps \to 0}{\longrightarrow} & \int \frac{\delta F}{\delta \mu_a} \ud T_{\frac{1}{a}}\#\bar{\mu}  +  \int \langle \big(\nabla \frac{\delta F}{\delta \mu_a}\big) \circ T_{\frac{1}{a}}, \partial_\eps T_{\frac{1}{a+\eps}}\vert_{\eps = 0} \rangle \ud \bar{\mu}\\
    =& \frac{1}{a} \int \frac{\delta F}{\delta \mu_a} \ud \mu_a - \frac{1}{a^2} \int \langle \big(\nabla \frac{\delta F}{\delta \mu_a}\big) \circ T_{\frac{1}{a}}, (x-x_0) \rangle \ud \bar{\mu}\\
    =& \frac{1}{a}\bigg[\int \frac{\delta F}{\delta \mu_a} \ud \mu_a - \int \langle \nabla \frac{\delta F}{\delta \mu_a},(x-x_0) \rangle \ud\mu_a \bigg]\\
    =& \frac{m_{\mu_a}}{a}m'_F(\mu_a)
\end{align*}
Thus $F(\mu_a)=F(\mu),$ for all $\mu \in \mc{M}(\R^d)\setminus \{ O_\mc{M} \}$ and all $a>0$ is equivalent to $m_F'(\mu_a)=0, \forall \mu \in \mc{M}(\R^d)\setminus \{ O_\mc{M} \}, \; \forall a>0$, which is equivalent to $m'_F(\mu) = 0, \quad \forall \mu \in \mc{M}(\R^d)\setminus \{O_\mc{M} \}$.  

We now prove the equivalence between $2.$ and $3.$. 
Given the $\WOP$ gradient $\nabla_{\WOP}F(\mu) = \big( m'_{F}(\mu), u_{F}(\mu) \big)$, the gradient flow equation is given by the weak partial differential equation integrated against the set $\mc{C}_c^{\infty}(\R^d \times ]0,\+\infty[)$
\[
\partial_t \nu_t =-\frac{m'_{F}(\nu_t)}{m_{\nu_t}}\nu_t + \dive\big(u_{F}(\nu_t) \nu_t \big).
\]
We prove that $m_{\nu_t}$ is weakly differentiable and the weak derivative is $-m'_{F}(\nu_t)$. Indeed let $ ( \phi_n(x) )_n$ be a sequence of functions in $\mc{C}^{\infty}_c(\R^d)$ satisfying 
\begin{align*}
& \phi_n(x) \underset{n\to+\infty}{\longrightarrow} 1 & \sup\limits_{n \in \mb{N}} \Vert \phi_n \Vert_{\infty} < +\infty\\
& \nabla\phi_n(x) \underset{n\to+\infty}{\longrightarrow} 0 & \sup\limits_{n \in \mb{N}} \Vert \nabla\phi_n \Vert_{\infty} < +\infty
\end{align*}
Let $\psi(t) \in \mc{C}_c^{\infty}(\R^\star_+)$, then $\phi_n(x)\psi(t) \in \mc{C}^{\infty}_c(\R^d \times \R^\star_+)$, and by integrating the gradient flow equation against this set of functions:
\[
    \int_{\R^d \times \R^\star_+} \phi_n\psi \partial_t \nu_t \ud t = -\int_{\R^d \times \R^\star_+} \phi_n\psi \frac{m'_{F}(\nu_t)}{m_{\nu_t}} \ud\nu_t \ud t + \int_{\R^d \times\R^\star} \phi_n\psi \dive\big(u_{F}(\nu_t) \nu_t \big)
    \]
    or equivalently,
    \[
    -\int_{\R^\star} \partial_t\psi \bigg[ \int\phi_n \ud\nu_t \bigg] \ud t = -\int_{\R^\star} \psi m'_{F}(\nu_t) \bigg[ \int \phi_n(x) \ud \bar{\nu}_t \bigg] \ud t - \int_{\R^\star} \psi \bigg[ \int \nabla\phi_n u_{F}(\nu_t) \ud\nu_t \bigg] \ud t.
    \]
Using the dominated convergence theorem, we end up with 
\[
-\int_{\R^\star} \partial_t \psi \; m_{\nu_t} \ud t = -\int_{\R^\star} \psi m'_{F}(\nu_t) \ud t.
\]
Using the continuity of $\nu_t$, we conclude the third equivalence. 
\end{proof}

This result readily allows us to extend functionals on $\mc{P}_2(\R^d)$ to functionals on $\mc{M}(\R^d)$ in a way that preserve the $W_2$ flow gradient.

\begin{corollary}[Extension of $\mc{P}_2$ functionals]\label{cor:extension}

Let $F$ be a functional on $\mc{P}_2(\R^d)$, we define on $\mc{M}(\R^d)$ the following functional 
\[
\tilde{F}(\mu)=F(T_{m_\mu}\#\bar{\mu}).
\]
Then $\tilde{F}$ and $F$ coincide on $\mc{P}_2(\R^d)$, the $\WOP$ gradient-flow of $\tilde{F}$ has constant mass and coincide with the $W_2$ gradient-flow of $F$ for probability measures.
\end{corollary}

For example, the Boltzmann entropy functional on $\mc{P}_2(\R^d)$ is given by
\[
E(\mu)=
\left\{ \begin{array}{ll} \int \log(\mu) \ud\mu & \text{ if } \mu\ll dx\\
+\infty & \text{ otherwise}\\
\end{array}\right. \quad \forall \mu \in \mc{P}_2(\R^d),
\]
The gradient flow of $E$ for the $W_2$ metric is the heat equation \cite{jordan1998variational}
\begin{equation}\label{eq:GF_Boltzmann1}
\partial_t \nu_t = \Delta \nu_t.
\end{equation}
The extension of the Boltzmann entropy for the $\WOP$ metric is given by
\[
\tilde{E}(\mu) = \left\{ \begin{array}{ll} \int \log(\bar{\mu}) \ud\bar{\mu} -d\log(m_\mu) & \text{ if } \mu\ll dx\\
+\infty & \text{ otherwise}\\
\end{array}\right. \quad \forall \mu \in \mc{M}(\R^d),
\]
The $\WOP$ gradient flow of $\tilde{E}$ satisfies the following equation
\begin{equation}\label{eq:GF_Boltzmann2}
\partial_t \nu_t = \frac{1}{m_{\nu_t}^2}\Delta \nu_t,
\end{equation}
which coincides with the heat equation~\eqref{eq:GF_Boltzmann1} when the initial measure of the gradient flow $\mu_0$ is a probability measure.

The following corollary gathers other examples of functionals and their $\WOP$ gradients. 
\begin{corollary}
Denote the $\WOP$ gradient of $F$ at $\mu$ by 
\[
\nabla_{\WOP}F(\mu)=(u_F,m_F')(\mu)
\]
For $F(\mu)=m_\mu$, the gradient is given by
\[
m'_F=1 \quad\text{and}\quad u_F(x)=-\frac{x-x_0}{m_\mu}.
\]
For $F$ such that $F(\mu)=F(\bar\mu)$ for all $\mu$, the gradient satisfies
\[
m'_F=-\frac{1}{m_\mu}\int\langle \nabla\frac{\delta F(\bar\mu)}{\delta \bar\mu},(x-x_0)\rangle\ud \bar\mu \quad\text{and}\quad u_F(x)=\frac{1}{m_\mu}\left[\nabla\frac{\delta F(\mu)}{\delta \mu}-m'(x-x_0)\right].
\]
In particular, if $F(\mu)=M_{x_0}(\bar{\mu})/2$, then
\[
m'_F=-\frac{M_{x_0}(\bar{\mu})}{m_\mu} \quad\text{and}\quad u_F(x) =\frac{1}{m_\mu^2}\left[(1+M_{x_0}(\bar{\mu}))(x-x_0)\right].
\]
and for $F(\mu)=m_\mu^2(1+M_{x_0}(\bar{\mu}))/2$, using $\nabla(fg)=f\nabla g+ g\nabla f$, we have
\[
m'_F=m_\mu \quad\text{and}\quad u_F(x)=0.
\]
Similarly, the gradient of $F(\mu)=m_\mu^2 M_{x_0}(\bar{\mu})/2$ is
\[
m'_F=0 \quad\text{and}\quad u_F(x)=x-x_0.
\]
\end{corollary}

\section{Barycenters}\label{sec:barycenter}

Let us now turn to barycenters in the space $(\mc{M}_K(\R^d),\WOP)$.
Recall that for a probability measure $\mb{P}=\sum_{i=1}^n\lambda_i\delta_{\mu_i}$ over $(\mc{M}(\R^d),\WOP)$, its barycenter $\mu_B\in\mc{M}(\R^d)$ is defined as the minimizer of
    \[
    \mu\mapsto V(\mu):=\sum_i\lambda_i \WOP_{x_0}^2(\mu,\mu_i)=\int \WOP^2(\mu,\tilde\mu)\ud\mb{P}(\tilde\mu).
    \]
Since $\WOP$ extends the Wasserstein distance, it is expected that there is a simple relationship between barycenters for the Wasserstein distance and for $\WOP$.
This is the content of the following result.

\begin{theorem}[Barycenters]
Let $\mb{P}=\sum_i\lambda_i\delta_{\mu_i}\in\mc{P}(\mc{M}(\R^d))$.
   Then its barycenter $\mu_B$ on $(\mc{M}(\R^d),\WOP_{x_0})$ has mass $m_{\mu_B}=\sum_i\lambda_im_{\mu_i}$ and $\bar\mu_B$ is the Wasserstein barycenter of $\sum_i\lambda_im_{\mu_i}\delta_{\bar\mu_i}$.
   In particular, although the value of $\WOP(\mu_0,\mu_1)$ depends on $x_0$, the geodesics between $\mu_0$ and $\mu_1$ do not depend on $x_0$.
\end{theorem}

Barycenters in this space do not depend on $x_0$ and have a simple form.
\begin{theorem}[Barycenters]
    Let $\mb{P}=\sum_i\lambda_i\delta_{\mu_i}\in\mc{P}(\mc{M}(\R^d))$.
    Its barycenter $\mu_B\in\mc{M}(\R^d)$ is defined as the minimizer of
    \[
    \mu\mapsto V(\mu):=\sum_i\lambda_i \WOP^2(\mu,\mu_i)=\int \WOP^2(\mu,\tilde\mu)\ud\mb{P}(\tilde\mu).
    \]
   Then $\mu_B$ does not depend on $x_0$.
   Moreoever, $m_{\mu_B}=\sum_i\lambda_im_{\mu_i}$ and $\bar\mu_B$ is the Wasserstein barycenter of $\sum_i\lambda_im_{\mu_i}\delta_{\bar\mu_i}$.
\end{theorem}
\begin{proof}
    Using \eqref{eq:defbis}, the barycenter $\mu_B$ thus minimizes
    \[
    \mu\mapsto\sum_i\lambda_i (m_\mu-m_{\mu_i})^2+ \sum_i\lambda_i (m_\mu-m_{\mu_i}) \big( m_\mu M_{x_0}(\bar\mu)-m_{\mu_i}M_{x_0}(\bar\mu_i) \big) + m_\mu \sum_i\lambda_i m_{\mu_i} W_2^2(\bar\mu,\bar\mu_i).
    \]
    Set $\tilde\lambda_i=\lambda_im_{\mu_i}$ for $i=1,\dots,n$, $\tilde\lambda_0=m_\mu-\sum_{i=1}^n\tilde\lambda_i$ and 
    \[
    \bar x= \frac{\sum_{i=0}^n\tilde\lambda_ix_i}{\sum_{i=0}^n\tilde\lambda_i}=\sum_{i=0}^n\frac{\tilde\lambda_ix_i}{m_\mu}
    \]
    Denote by $\gamma$ the coupling of $\bar\mu,\bar\mu_1,\dots,\bar\mu_n$ obtained by gluing the optimal couplings between $\bar\mu$ and $\bar\mu_i$.
    Then,
    \begin{align*}
\sum_i\lambda_i \big(m_{\mu_i} W_2^2(\bar\mu,\bar\mu_i)+(m_\mu-m_{\mu_i})M_{x_0}(\bar\mu) \big)  &=\int\sum_{i=1}^n\tilde\lambda_i|x_i-x|^2+\tilde\lambda_0|x_0-x|^2\ud\gamma(x,x_1,\dots,x_n)\\
&\ge \int\sum_{i=1}^n\tilde\lambda_i|x_i-\bar x|^2+\tilde\lambda_0|x_0-\bar x|^2\ud\gamma(x,x_1,\dots,x_n)\\
&= \sum_i\lambda_i \big( m_{\mu_i} W_2^2(\bar\mu,\bar\mu_i)+(m_\mu-m_{\mu_i}) M_{x_0}(\bar\mu) \big)
    \end{align*}
    Now, using that $\sum_{i=0}^n\frac{\tilde\lambda_i}{m_\mu}|x_i-\bar x|^2=\sum_{i=1}^n\frac{\tilde\lambda_i}{m_\mu}|x_i-x_0|^2-|\bar x-x_0|^2$, we get
    \[
\sum_i\lambda_im_{\mu_i} W_2^2(\bar\mu,\bar\mu_i)+(m_\mu-\sum_i\lambda_im_{\mu_i})M_{x_0}(\bar\mu)=\sum_i\lambda_im_{\mu_i} M_{x_0}(\bar\mu_i)-m_\mu M_{x_0}(\bar\mu) 
    \]
    Lastly,
    \begin{align*}
m_\mu^2 M_{x_0}(\bar\mu) &=m_\mu^2\int |\bar x-x_0|^2\ud\gamma(x,x_1,\dots,x_n)\\
&=\int|\sum_{i=1}^n\tilde\lambda_i(x_i-x_0)|^2\ud\gamma(x,x_1,\dots,x_n)
    \end{align*}
    Therefore, for $\bar\mu$ fixed, $m_{\mu}$ that minimizes $V(\mu)$ satisfies $m_\mu=\sum_i\lambda_im_{\mu_i}$.
    Indeed,
    \begin{align*}
V(\mu)&=\sum_i\lambda_i (m_\mu-m_{\mu_i})^2 - \sum_i\lambda_i (m_\mu-m_{\mu_i})m_{\mu_i}M_{x_0}(\bar\mu_i) + m \big(\sum_i\lambda_im_{\mu_i} M_{x_0}(\bar\mu_i)-m_\mu M_{x_0}(\bar\mu) \big) \\
&=\sum_i\lambda_i (m_\mu-m_{\mu_i})^2 - m_\mu^2M_{x_0}(\bar\mu)+ \sum_i\lambda_i m_{\mu_i}^2M_{x_0}(\bar\mu_i)\\
&=\sum_i\lambda_i (m_\mu-m_{\mu_i})^2 \underbrace{-\int|\sum_{i=1}^n\tilde\lambda_i(x_i-x_0)|^2\ud\gamma(x,x_1,\dots,x_n)+ \sum_i\lambda_i m_{\mu_i}^2M_{x_0}(\bar\mu_i)}_{\text{does not depend on }m_\mu}.
    \end{align*}
    For the last claim of the proof, remark than $V(\mu)$ depends on $\bar\mu$ only through $\int|\sum_{i=1}^n\tilde\lambda_i(x_i-x_0)|^2\ud\gamma(x,x_1,\dots,x_n)$, the optimal coupling of which does not depend on $x_0$.
\end{proof}

\section{Further developments}\label{sec:further}

In this section we study some limitations when extending the $2$-Wasserstein metric, as well as the impossibility to have the Entropy-Transport formulation extending the $2$-Wasserstein metric. 
We also briefly present an attempt to extend the $\WOP$ metric to any exponent $p \ge 1$.
The following property is a limitation for any metric on $\mc{M}(\R^d)$ generalizing the Wasserstein metric.

\paragraph{Impossibility result}
We start with an observation on the distance between the null measure and the space of probability measures.

The $\WOP$ distance of any measure of  $\mc{P}_2(\R^d)$ to the null measure $0_\mc{M}$ is not constant.
In fact, for Dirac measures $\delta_x$,
\[ 
\WOP(\delta_x,0_{\mc{M}}) = \sqrt{1 + \vert x - x_0 \vert^2 } .
\]

\begin{property}[Distance between $\mc{P}(\R^d)$ and the null measure]\label{th:impossibility1}
This is no metric $D(\mu,\nu)$ on $\mc{M}(\mc{X})$ with $\mc{X} \subset \R^d$ satisfying simultaneously the following properties.
\begin{enumerate}
    \item $D$ coincides with the 2-Wasserstein metric on $\mc{P}_2(\mc{X})$, i.e. $D(\mu,\nu)= W_2(\mu,\nu)$ for all $(\mu,\nu) \in \mc{P}_2(\mc{X})^2$\label{it:Dgeneral},
    \item $D(0,\mu)$ is bound, i.e. there exist $\alpha>0$ such that $D(0,\mu) \le \alpha, \quad \forall \mu \in \mc{P}(\mc{X})$,
    \item $\mc{X}$ is unbounded 
\end{enumerate}
\end{property}
Thus if one wants to generalize the Wasserstein metric and work for instance with Gaussian measures, one must drop the equidistance of measures in $\mc{P}^2(\R^d)$ to the null measure.
In comparison the $\HK$ metric does not satisfy point $1.$ but the points $2$ and $3$ hold.

\begin{proof}
Let $\mc{X} \subset \R^d$ and suppose that $(\mc{M}(\mc{X}),D)$ is a metric space satisfying the generalizing condition (\ref{it:Dgeneral}).
Then for any measures $(\mu,\nu) \in \mc{P}_2(\mc{X})^2$ we have, using the triangle inequality:
\[
W_2(\mu,\nu) = D(\mu,\nu) \le D(\mu,0_{\mc{M}}) + D(0_\mc{M},\nu) \le 2\alpha.
\]
This bounding condition on $W_2$ is equivalent to say that $\diam( \mc{X} )  \le 2\alpha$. Indeed for the implication $(\Rightarrow)$ we can take  $\mu = \delta_x$ and $\nu = \delta_y$ with $(x,y) \in \mc{X}^2$, then we have :
$$ W_2(\delta_x,\delta_y)= \vert x-y \vert \le 2\alpha $$

\noindent
The implication $(\Leftarrow)$ can be proven as follow :
\[
W_2(\mu,\nu) = \sqrt{\inf_{\pi \in \Pi(\mu,\nu)} \int \vert x-y \vert^2 \ud\pi} \le \sqrt{\inf_{\pi \in \Pi(\mu,\nu)} \int \diam(\mc{X})^2 \ud\pi} =\diam(\mc{X}) \le 2 \alpha.
\]
And so $\mc{X}$ is bounded which contradicts our assumption.
\end{proof}

We here justify the development of our metric outside of the broad framework of Entropy-Transport formulation.

Let us first recall the definition of Csiszar $f$-divergences and the $\ET$ formulation in a very general case.
Let $f:\R_+ \longrightarrow [0,+\infty]$ be a convex, lower semi-continuous function such that $f(1)=0$, we define $f_{\infty}'(1) = \lim\limits_{t \to +\infty} \frac{f(t)}{t}$. Let $\mu,\nu \in \mc{M}(\R^d)$, the Lebesgue decomposition of $\mu$ with respect to $\nu$ is $\mu = \frac{d\mu}{d\nu}\nu + \nu^\perp$.
Then the $f$-divergence is given by
\[
D_f(\mu \vert \nu) = \int f(\frac{d\mu}{d\nu}) \ud \nu + f_{\infty}'(1).m_{\nu^{\perp}},
\]
With the convention that $+\infty.0=0$.
Now to define the $\ET$ in a  general setting, let $f: \R_+ \to [0,+\infty]$ such that $f$ is convex, lower semi-continuous and with $f(1)=0$, let $c(x,y) : \R^d \times \R^d \to [0,+\infty]$ be lower semi-continuous such that $c(x,x)=0$.  Then the entropy transport is given by
\[
\ET(\mu,\nu) = \inf\limits_{\gamma \in \mc{M}(\R^d \times \R^d)} \bigg\{ D_{f}(P_0\#\gamma \vert \mu) + D_{f}(P_1\#\gamma \vert \nu) + \int c(x,y) \ud\gamma(x,y) \bigg\},
\]
Where $P_0\#\gamma$ is the first marginal of $\gamma$ and $P_1\#\gamma$ is the second marginal of $\gamma$.\\
\begin{theorem}[Impossibility for $\ET$ metrics to extend $W_2$]
If $\min \big( f(0),f_{\infty}'(1) \big)<+\infty$, then $\ET$ cannot be a metric extending $W_2$ for positive measures on $\R^d$.
Otherwise, if $\min \big( f(0),f_{\infty}'(1) \big)=+\infty$, then $\ET(\mu,0_\mc{M})=+\infty$ for all $\mu\in\mc{M}(\R^d)\setminus\{0_{\mc{M}}\}$. 
\end{theorem}

\begin{proof}
for $\mu \in \mc{P}_2(\R^d)$, we have the following bound:
\begin{equation}\label{eq:ETbound}
\ET(\mu,0_\mc{M}) \le \min \big( f(0),f_{\infty}'(1) \big).
\end{equation}
Indeed, by taking the sub-optimal transportation plan $\gamma=0_\mc{M}$ we have
\[ \ET(\mu,0_\mc{M}) \le D_f(0_\mc{M},\mu) + D_f(0_\mc{M},0_\mc{M}) + \int c(x,y) \ud 0_\mc{M} = f(0).m_\mu =f(0)
\]
And by taking the sub-optimal transport plan $\gamma=(Id,Id)\#\mu$ we obtain
\[ \ET(\mu,0_\mc{M}) \le D_f(\mu,\mu) + D_f(\mu,0_\mc{M}) + \int c(x,x) \ud\mu = f(1).m_\mu+f(0).0+f_{\infty}'(1).m_\mu =f_{\infty}'(1).
\]
Using our bound~\eqref{eq:ETbound} with Theorem~\ref{th:impossibility1}, we conclude that no metric based on the Entropy-Transport formulation can generalize $W_2$ for measures in $\R^d$.
\end{proof}

\paragraph{Extending $W_p$}
We end this section with a discussion on extending the idea to the $W_p$ metrics.

    We can define the $\WOP_p$ metric for $p \ge 1$ as follows,
    \[
     \WOP_{p}^p(\mu,\nu) = \vert m_\mu- m_\nu \vert^p + W_p^p( T_{m_\mu}\#\bar{\mu} ,T_{m_\nu}\#\bar{\nu} )
    \]
    
    And the space $\mc{M}_{K,p}(\R^d)$ as the set of positive measures such that the p-order moment at $x_0$, noted $M_{x_0,p}(\mu)$, verifies the following inequality
    
    \[
     M_{x_0,p}(\mu) \le Km_\mu
    \]
    
    Then $\WOP_{p}$ satisfy almost every properties verified by $\WOP$ :
    \begin{itemize}
    \item $\WOP_p$ is a metric on $\mc{M}(\R^d)$
    \item $\WOP_p$ generalizes $p$-Wasserstein
    \item $(\mc{M}_{K,p}(\R^d),\WOP_p)$ is complete for all $K>0$
    \item $\WOP_p$ metrizes the weak convergence + convergence of the $p^{th}$-order moments in $\mc{M}_{K,p}(\R^d)$ 
    \item $\WOP_{p}$ verifies the 1-homogeneity
    \item $\WOP_p$ is a geodesic metric where the geodeosic between two measures $\mu_0$ and $\mu_1$ is :
    \[
    \mu_t = m_t.\eta_t
    \]
    Where $m_t = (1-t)m_{\mu_0} + tm_{\mu_1}$ and $\eta_t$ is the geodesic in $(\mc{P}_2(\R^d),W_p)$ between the measures $T_{m_{\mu_0}}\#\bar{\mu}_0$ and  $T_{m_{\mu_1}}\#\bar{\mu}_1$
    \end{itemize}
    Note that the geodesics and the barycenters depend on the choice of the reference point $x_0$ in the case $p \neq 2$.

\printbibliography

\end{document}